\documentclass[11pt,reqno]{amsart}

\usepackage{amsmath, amsthm, amssymb}
\usepackage{amsfonts}
\usepackage[ansinew]{inputenc}
\usepackage[dvips]{epsfig}
\usepackage{graphicx}
\usepackage[english]{babel}

\usepackage{cite}
\usepackage{graphicx}
\usepackage{amscd}
\usepackage{xcolor}
\usepackage{bm}
\usepackage{enumerate}

\usepackage{verbatim}
\usepackage{hyperref}
\usepackage{amstext}
\usepackage{latexsym}

\let\oldsqrt\sqrt
\def\sqrt{\mathpalette\DHLhksqrt}
\def\DHLhksqrt#1#2{%
\setbox0=\hbox{$#1\oldsqrt{#2\,}$}\dimen0=\ht0
\advance\dimen0-0.2\ht0
\setbox2=\hbox{\vrule height\ht0 depth -\dimen0}%
{\box0\lower0.4pt\box2}}

\allowdisplaybreaks

\newcommand{\loglap}{L_{\text{\tiny $\Delta \:$}}\!}

\newcommand{\R}{\mathbb{R}} 
\newcommand{\N}{\mathbb{N}} 

\newcommand{\dist}{\textnormal{dist}} 
\newcommand{\diam}{\textnormal{diam}} 
\newcommand{\osc}{\textnormal{osc}} 

\DeclareMathOperator{\id}{\rm id}

\renewcommand{\phi}{\varphi}

\newcommand{\cE}{{\mathcal E}}
\newcommand{\cF}{{\mathcal F}}

\newcommand{\cH}{{\mathcal H}}

\newcommand{\cR}{{\mathcal R}}

\newcommand{\cV}{{\mathcal V}}

\newcommand{\smhoo}{\!\text{\tiny $\cH^0_0(\Omega)$}}

\newcommand{\eps}{\varepsilon}

\theoremstyle{definition}
\newtheorem{defi}{Definition}[section]
\newtheorem{remark}[defi]{Remark}

\theoremstyle{plain} 
\newtheorem{thm}[defi]{Theorem}
\newtheorem{prop}[defi]{Proposition}
\newtheorem{lemma}[defi]{Lemma}

\newtheorem{cor}[defi]{Corollary}

\theoremstyle{definition}

\numberwithin{equation}{section}

 \title[Small order asymptotics for the fractional Laplacian]{Small order asymptotics of the Dirichlet eigenvalue problem for the fractional Laplacian}

\author[]
{Pierre Aime Feulefack, Sven Jarohs, Tobias Weth}

\address{Goethe-Universit\"{a}t Frankfurt, Institut f\"{u}r Mathematik.
Robert-Mayer-Str. 10, D-60629 Frankfurt, Germany.}

\email{feulefac@math.uni-frankfurt.de}

\address{Goethe-Universit\"{a}t Frankfurt, Institut f\"{u}r Mathematik.
Robert-Mayer-Str. 10, D-60629 Frankfurt, Germany.}

\email{jarohs@math.uni-frankfurt.de}

\address{Goethe-Universit\"{a}t Frankfurt, Institut f\"{u}r Mathematik.
Robert-Mayer-Str. 10, D-60629 Frankfurt, Germany.}

\email{weth@math.uni-frankfurt.de}

\begin{document}

\begin{abstract}
We study the asymptotics of Dirichlet eigenvalues and eigenfunctions of the fractional Laplacian $(-\Delta)^s$ in bounded open Lipschitz sets in the small order limit $s \to 0^+$. While it is easy to see that all eigenvalues converge to $1$ as $s \to 0^+$, we show that the first order correction in these asymptotics is given by the eigenvalues of the logarithmic Laplacian operator, i.e., the singular integral operator with symbol $2\log|\xi|$. By this we generalize a result of Chen and the third author which was restricted to the principal eigenvalue. 
Moreover, we show that $L^2$-normalized Dirichlet eigenfunctions of $(-\Delta)^s$ corresponding to the $k$-th eigenvalue are uniformly bounded and converge to the set of $L^2$-normalized eigenfunctions
of the logarithmic Laplacian. In order to derive these spectral asymptotics, we establish new uniform regularity and boundary decay estimates for Dirichlet eigenfunctions for the fractional Laplacian. As a byproduct, we also obtain corresponding regularity properties of eigenfunctions of the logarithmic Laplacian.
\end{abstract}

\maketitle

{\footnotesize

  \textit{Keywords:} Fractional Laplacian, small order expansion, logarithmic Laplacian, uniform regularity.\\

  \textit{MSC2010:} 35R11, 45C05, 26A33.
}

\section{Introduction}\label{intro}

Fueled by various applications and important links to stochastic processes and partial differential equations, the interest in nonlocal operators and associated Dirichlet problems has been growing rapidly in recent years. In this context, the  fractional Laplacian has received by far the most attention, see e.g. \cite{BLMH,BK,caffarelli2007extension,RefG,BB99,BKK08,RefCBV,RefK2,Ros14} and the references therein. We recall that, for compactly supported functions $u: \R^N \to  \R$ of class $C^2$ and $s \in (0,1)$, the fractional Laplacian $(-\Delta)^s$ is well-defined by
        \begin{equation}\label{c-n-s-constant}
(-\Delta)^s u(x)=C_{N,s}\lim_{\epsilon\to 0^+}\int_{\R^N\setminus B_{\epsilon}(x)}\frac{u(x)-u(y)}{|x-y|^{N+2s}}\ dy,\quad \text{where}
\quad C_{N,s}= s4^{s}\frac{\Gamma(\frac{N}{2}+s)}{\pi^{\frac{N}{2}}\Gamma(1-s)}.
\end{equation} 
The normalization constant $C_{N,s}$ is chosen such that $(-\Delta)^s$ is equivalently given by
\begin{equation}
  \label{eq:fourier-representation}
\cF\bigl((-\Delta)^su\bigr) = |\cdot|^{2s}\cF{u},    
\end{equation}
where, here and in the following, $\cF$ denotes the usual Fourier transform. We emphasize that the fractional Laplacian is an operator of order $2s$ and many related regularity properties -- in particular of associated eigenfunctions -- rely on this fact.

The present paper is concerned with the small order asymptotics $s \to 0^+$ of the Dirichlet eigenvalue problem   
 \begin{equation}\label{eq2}
	\begin{split}
	\quad\left\{\begin{aligned}
		(-\Delta)^s\phi_s &= \lambda \phi_s && \text{ in\ \  $\Omega$,}\\
		\phi_s &=  0             && \text{ in\ \  }\Omega^c,	 
	\end{aligned}\right.
	\end{split}
	\end{equation}
        where $\Omega \subset \R^N$ is a bounded open set with Lipschitz boundary and $\Omega^c:= \R^N \setminus \Omega$. It is well known (see \cite[Proposition 9]{SV13} or \cite[Proposition 3.1]{RefG}) that, for every $s \in (0,1)$, \eqref{eq2} admits an ordered sequence of eigenvalues  
\begin{equation}
  \label{eq:eigenvalues-fractional}
\lambda_{1,s}< \lambda_{2,s} \le \lambda_{3,s} \le  \dots   
\end{equation}
with $\lambda_{k,s} \to \infty$ as $k \to \infty$ and a corresponding $L^2$-orthonormal basis of eigenfunctions $\phi_{k,s}$, $k \in \N$. Moreover, $\phi_{1,s}$ is unique up to sign and can be chosen as a positive function.

The starting point of the present work is the basic observation that
\begin{equation}
  \label{eq:fractional-laplace-first-easy-conv}
  (-\Delta)^s u \to u \qquad \text{as $s \to 0^+$ for every $u \in C^2_c(\R^N)$,}
	\end{equation}
        which readily follows from \eqref{eq:fourier-representation} and standard properties of the Fourier transform (see also \cite[Proposition 4.4]{RefV0}. Similarly, we have
\begin{equation}
  \label{eq:fractional-laplace-quadratic-form-first-easy-conv}
\cE_s(u,u) \to \|u\|_{L^2(\R^N)}^2 \qquad \text{as $s \to 0^+$ for every $u \in C^1_c(\R^N)$,}
	\end{equation}
where $\cE_s$ denotes the quadratic form associated with $(-\Delta)^s$ given by 
        $$
(u,v) \mapsto \cE_s(u,v)=\frac{C_{N,s}}{2}\displaystyle\int_{\R^N}\int_{\R^N}\frac{\big(u(x)-u(y)\big)\big(v(x)-v(y)\big)}{|x-y|^{N+2s}}\ dxdy.
$$
We remark that these convergence properties in the limit $s \to 0^+$ extend to a non-Hilbertian setting of quasilinear operators where the Fourier transform cannot be employed, see e.g. \cite{cianchi-et-al} and the references therein. It is not difficult to deduce from (\ref{eq:fractional-laplace-first-easy-conv}) that 
\begin{equation}
  \label{eq:eigenvalues-convergence-to-one}
\lambda_{k,s} \to 1 \qquad \text{as $s \to 0^+$ for all $k \in \N$,}  
\end{equation}
see Section~\ref{section2} below for details. However, there is no straightforward approach to obtain the asymptotics of associated eigenfunctions as $s \to 0^+$ since, as a consequence of (\ref{eq:fractional-laplace-first-easy-conv}) and (\ref{eq:fractional-laplace-quadratic-form-first-easy-conv}), no uniform regularity theory is available for the fractional Laplacian $(-\Delta)^s$ in the case where $s$ is close to zero. For general bounded open sets with Lipschitz boundary, the only available result regarding these asymptotics is contained in \cite{RefT1}, where Chen and the third author introduced the logarithmic Laplacian operator $\loglap$ to give a more detailed description of the first eigenvalue $\lambda_{1,s}$ and the corresponding eigenfunction $\phi_{1,s}$ as $s \to 0^+$. On compactly supported Dini continuous functions, the operator $\loglap$ is pointwisely given by 
\begin{equation}\label{int-log}
	\loglap u(x) = C_N\int_{\R^N}\frac{u(x)1_{B_1(x)}(y)-u(y)}{|x-y|^N}\ dy +\rho_Nu(x),
\end{equation}
where $C_N = \pi^{-\frac{N}{2}}\Gamma(\frac{N}{2})$, and $\rho_{N} = 2\log2 + \psi(\frac{N}{2})-\gamma$. Here, $\psi = \frac{\Gamma'}{\Gamma}$ denotes the Digamma function, and $\gamma = -\Gamma'(1)$ is the Euler-Mascheroni constant. 

We note two key properties of the operator $\loglap$ shown in \cite{RefT1}. If $u \in C^\beta_c(\R^N)$ for some $\beta>0$, then 
\begin{equation}\label{Fourier-log}
\mathcal{F}(\loglap u) = 2\log|\xi|\mathcal{F}(u)(\xi)~~~~~~\text{ for a.e.}~~ \xi\in \R^N,
\end{equation}
and
\begin{equation}
  \label{eq:chen-weth-limit-log}
\frac{d}{ds}\Bigl|_{s=0}(-\Delta)^s u = \lim_{s \to 0^+}\frac{(-\Delta)^s u - u}{s}= \loglap u \qquad \text{in $L^p(\R^N)$ for $1<p \le \infty$.} 
\end{equation}
Hence, $\loglap$ arises as a formal derivative of fractional Laplacians at $s = 0$. As a consequence of \eqref{Fourier-log}, $\loglap$ is an operator of {\em logarithmic order}, and it belongs to a class of weakly singular integral operators having an intrinsic scaling property. Operators of this type have also been studied e.g. in \cite{RefE,RefJ2,RefJ1,RefMA}. The operator $\loglap$ also arises in a geometric context of the $0$-fractional perimeter studied recently in \cite{DNP}.

Using \eqref{eq:chen-weth-limit-log} and related functional analytic properties, it has been shown in \cite[Theorem 1.5]{RefT1} that
\begin{equation}
  \label{eq:RefT1-spectral-convergence}
\frac{\lambda_{1,s}-1}{s} \to \lambda_{1,L} \quad \text{and}\quad \phi_{1,s} \to \phi_{1,L} \quad \text{in $L^2(\Omega)$} \qquad \text{as $s \to 0^+$,}  
\end{equation}
where $\lambda_{1,L}$ denotes the principal eigenvalue of the Dirichlet eigenvalue problem 
\begin{equation}\label{eq-eigenvalue-loglap}
	\begin{split}
	\quad\left\{\begin{aligned}
		\loglap u &= \lambda u&& \text{ in\ \  $\Omega$,}\\
		u &=  0             && \text{ in\ \  }\Omega^c,	 
	\end{aligned}\right.
	\end{split}
	\end{equation}
        and $\phi_{1,L}$ denotes the corresponding (unique) positive $L^2$-normalized eigenfunction. Here we note that we consider both (\ref{eq2}) and (\ref{eq-eigenvalue-loglap}) in a suitable weak sense which we will make more precise below. 

The main aim of the present paper is twofold. First, we wish to improve the $L^2$-convergence $\phi_{1,s} \to \phi_{1,L}$ in \eqref{eq:RefT1-spectral-convergence}. For this, new tools are needed in order to overcome the lack of uniform regularity estimates for the fractional Laplacian $(-\Delta)^s$ for $s$ close to zero. Secondly, we wish to extend the convergence result from \cite{RefT1} to higher eigenvalues and eigenfunctions. Due to the multiplicity of eigenvalues and eigenfunctions for $k\geq 2$, this also requires a new approach based on the use of Fourier transform in combination with the Courant-Fischer characterization of eigenvalues.

In order to state our main results, we need to introduce some notation regarding the weak formulations of (\ref{eq2}) and (\ref{eq-eigenvalue-loglap}). For the weak formulation of (\ref{eq2}), we consider the standard Sobolev space
\begin{equation}\label{wsp-d-intro}
\cH^{s}_{0}(\Omega):=\left\{u\in H^{s}(\R^N)\;:\;u\equiv 0 \text{ on  $\Omega^c$}\right\} 
\end{equation}
and we call $\phi \in \cH^s_0(\Omega)$ an eigenfunction of \eqref{eq2} corresponding to the eigenvalue $\lambda$ if
$$
\cE_s(\phi,v) = \lambda \int_{\Omega} \phi v \ dx \qquad \text{for all $v \in \cH^s_0(\Omega)$.}
$$
For the weak formulation of (\ref{eq-eigenvalue-loglap}), we follow \cite{RefT1} and define the space
\begin{equation}\label{log-space-intro}
\cH^0_0(\Omega):=\Big\{u\in L^2(\R^N)\;:\; u\equiv 0 \text{ on  $\Omega^c,\quad $} \langle u,u \rangle_{\smhoo}<+\infty \Big\}, 
\end{equation}
where the quadratic form $\langle \cdot,\cdot \rangle_{\smhoo}$ is given by 
\begin{equation}\label{scalar-prod-L-intro}
(u,v)\mapsto \langle u,v \rangle_{\smhoo} :=  \frac{C_N}{2}\iint_{\substack{{x,y\in \R^N}\\|x-y|< 1}}\frac{(u(x)-u(y))(v(x)-v(y))}{|x-y|^N}\ dxdy.
\end{equation}
A function $\phi \in \cH^0_0(\Omega)$ is called an eigenfunction of \eqref{eq-eigenvalue-loglap} corresponding to the eigenvalue $\lambda$ if
$$
\cE_L(\phi,v) = \lambda \int_{\Omega} \phi v\ dx \qquad \text{for all $v \in \cH^0_0(\Omega)$,}
$$
where 
\begin{equation}\label{EL-diff-sign-intro}
(u,v) \mapsto \cE_{L}(u,v) = \langle u,v \rangle_{\smhoo}-C_N\iint_{\substack{{x,y\in \R^N}\\|x-y|\ge 1}}\frac{u(x)v(y)}{|x-y|^N}\ dxdy + \rho_N\int_{\R^N}uv\ dx
\end{equation}
is the quadratic form associated with $\loglap$. For more details, see Section~\ref{section2} below and \cite{RefT1}.\\

The first main result of this paper now reads as follows.

\begin{thm}
\label{main-theorem-introduction}
Let $\Omega \subset \R^N$ be a bounded open set with Lipschitz boundary and let $k \in \N$. Moreover, for $s \in (0,\frac{1}{4})$, let $\lambda_{k,s}$ resp. $\lambda_{k,L}$ denote the $k$-th Dirichlet eigenvalue of the fractional and logarithmic Laplacian, respectively, and let $\phi_{k,s}$ denote an $L^2$-normalized eigenfunction. Then we have:  
\begin{enumerate}
\item[(i)] The eigenvalue $\lambda_{k,s}$ satisfies the expansion
\begin{equation}
  \label{eq:eigenvalue-expansion-intro}
\lambda_{k,s} = 1 + s \lambda_{k,L} + o(s) \qquad \text{as $s \to 0^+$.}
\end{equation}
\item[(ii)] The set $\{\phi_{k,s}\::\: s \in (0,\frac{1}{4}]\}$ is bounded in $L^\infty(\Omega)$ and relatively compact in $L^p(\Omega)$ for every $p<\infty$.
\item[(iii)] The set $\{\phi_{k,s}\::\: s \in (0,\frac{1}{4}]\}$ is
equicontinuous in every point $x_0 \in \Omega$ and therefore relative compact in $C(K)$ for any compact subset $K \subset \Omega$.  
\item[(iv)] If $\Omega$ satisfies an exterior sphere condition, then the set
  $\{\phi_{k,s}\::\: s \in (0,\frac{1}{4}]\}$ is relative compact in the space $C_0(\Omega) := \{u \in C(\R^N)\::\: u \equiv 0 \quad \text{in $\Omega^c$}\}$.
\item[(v)] If $(s_n)_n \subset (0,\frac{1}{4}]$ is a sequence with $s_n \to 0$ as $n \to \infty$, then, after passing to a subsequence, we have
  \begin{equation}
    \label{eq:convergence-eigenfunctions-main-theorem}
  \phi_{k,s_n} \to \phi_{k,L} \quad \text{as $n \to \infty$}
  \end{equation}
  in $L^p(\Omega)$ for $p < \infty$ and locally uniformly in $\Omega$, where $\phi_{k,L}$ is an $L^2$-normalized eigenfunction of the logarithmic Laplacian corresponding to the eigenvalue $\lambda_{k,L}$.\\[0.2cm]  
  If, moreover, $\Omega$ satisfies an exterior sphere condition, then the convergence in \eqref{eq:convergence-eigenfunctions-main-theorem} is uniform in $\overline \Omega$. 
\end{enumerate}
\end{thm}

Here and in the following, we identify the space $L^p(\Omega)$ with the space of functions $u \in L^p(\R^N)$ with $u \equiv 0$ on $\Omega$.

\begin{remark}
  \label{rem-main-thm-intro}
  \begin{enumerate}
  \item[(i)] Theorem~\ref{main-theorem-introduction} complements \cite[Theorem 1.5]{RefT1} by emphazising the relevance of higher Dirichlet eigenvalues and eigenfunctions of $\loglap$ for the spectral asymptotics of the fractional Laplacian as $s \to 0^+$. We note that upper and lower bounds for the Dirichlet eigenvalues $\lambda_{k,L}$ of the logarithmic Laplacian and corresponding Weyl type asymptotics in the limit $k \to + \infty$ have been derived in \cite{laptev-weth} and more recently in \cite{chen-veron}.
      \item[(ii)] The number $\frac{1}{4}$ in the above theorem is chosen for technical reasons, as it allows to reduce the number of case distinctions in the arguments. In the case $N \ge 2$, it can be replaced by any fixed number smaller than $1$, and in the case $N=1$ it can be replaced by any fixed number smaller than $\frac{1}{2}$. Since we are only interested in parameters $s$ close to zero in this paper, we omit the details of such an extension.  
  \item[(iii)] In the case where $\Omega$ is a bounded open Lipschitz set satisfying an exterior sphere condition and $k \in \N$ is fixed, the functions $\phi_{k,s}$, $s \in (0,\frac{1}{4}]$ satisfy a uniform decay condition in the sense that there exists a function $h_k \in C_0(\Omega)$ with the property that $|\phi_{k,s}| \le h_k$ in $\Omega$ for all $s \in (0,\frac{1}{4}]$. This is a rather direct consequence of Theorem~\ref{main-theorem-introduction}(iv). However, the optimal rate of the uniform boundary remains open. We conjecture that the function $h$ can be chosen with the property that
    $$
    h_k(x) \sim \bigl(- \ln \dist(x,\Omega^c)\bigr)^{-\tau} \qquad \text{as $x \to \partial \Omega$ for some $\tau>0$.}
    $$
  \end{enumerate}
\end{remark}
\smallskip

As noted already, the principal eigenvalue $\lambda_{1,s}(\Omega)$ admits, up to sign, a unique $L^2$-normalized eigenfunction which can be chosen to be positive. Hence Theorem~\ref{main-theorem-introduction} and \cite[Theorem 1.5]{RefT1} give rise to the following corollary.

\begin{cor}
  \label{first-cor-intro}
  Let $\Omega \subset \R^N$ be a bounded open set with Lipschitz boundary and let, for $s \in (0,\frac{1}{4}]$, $\phi_{1,s}$ denote the unique positive $L^2$-normalized eigenfunction of $(-\Delta)^s$ corresponding to the principal Dirichlet eigenvalue $\lambda_{1,s}$. Then we have 
  \begin{equation}
    \label{eq:convergence-first-eigenfunction-cor}
  \phi_{1,s} \to \phi_{1,L} \quad \text{as $s \to 0^+$}
  \end{equation}
  in $L^p(\Omega)$ for $p < \infty$ and locally uniformly in $\Omega$, where $\phi_{1,L}$ is the unique positive $L^2$-normalized eigenfunction of $\loglap$ corresponding to the principal Dirichlet eigenvalue $\lambda_{1,L}$.\\[0.2cm]  
  If, moreover, $\Omega$ satisfies an exterior sphere condition, then the convergence in \eqref{eq:convergence-first-eigenfunction-cor} is uniform in $\overline \Omega$. 
\end{cor}

As a further corollary of Theorem~\ref{main-theorem-introduction}, we shall derive the following regularity properties of eigenfunctions of the logarithmic Laplacian.

\begin{cor}
  \label{cor-regularity-loglap}
Let $\Omega \subset \R^N$ be a bounded open set with Lipschitz boundary, and let $\phi \in \cH^0_0(\Omega)$ be an eigenfunction of \eqref{eq-eigenvalue-loglap}. Then $\phi \in L^\infty(\Omega) \cap C_{loc}(\Omega)$. Moreover, if $\Omega$ satisfies an exterior sphere condition, then $\phi \in C_0(\Omega)$.    
\end{cor}

\begin{remark}
  \label{sec:introduction}
  If $\Omega$ is a bounded open Lipschitz set satisfying an exterior sphere condition, then we can deduce from the $L^\infty$-bound on $\phi$ given in Corollary~\ref{cor-regularity-loglap} and \cite[Theorem 1.11]{RefT1} that
  every eigenfunction $\phi \in \cH^0_0(\Omega)$ of \eqref{eq-eigenvalue-loglap}
  satisfies
  $$
  |\phi(x)|= O\Bigl(\bigl(-\ln \dist(x,\Omega^c)\bigr)^{-\tau}\Bigr) \qquad \text{as $x \to \partial \Omega$ for all $\tau \in (0,\frac{1}{2}).$}
  $$
  This motivates the conjecture in Remark~\ref{rem-main-thm-intro}(ii). In fact, we shall use the same barrier function as in \cite[Theorem 1.11]{RefT1} for $\loglap$ to prove Theorem~\ref{main-theorem-introduction}(iv), but we cannot use it in the same way since we are not able to derive uniform bounds on the difference
  \begin{equation}
    \label{eq:difference-loglap-frac}
\Bigl[  \loglap - \frac{(-\Delta)^s-\id}{s}\Bigr]\phi_{k,s}
  \end{equation}
  close to the boundary $\partial \Omega$ as $s \to 0^+$. 
\end{remark}
\smallskip

The paper is organized as follows. In Section  \ref{section2}, we collect preliminary results on the functional analytic setting. Moreover, we prove the asymptotic expansion~\eqref{eq:eigenvalue-expansion-intro} and the $L^2$-convergence property asserted in Theorem~\ref{main-theorem-introduction}(v). In Section~\ref{section-regularity-of-LD}, we prove the uniform $L^\infty$-bound on eigenfunctions as stated in Theorem~\ref{main-theorem-introduction}(ii). For this, we use a new technique based on the splitting of the integral over $\R^N$ on a small ball of radius  $\delta$ ($\delta$-decomposition) and apply known results and conditions associated to the newly obtained quadratic form as in \cite{RefM1,RefJ1}. We emphasize that this technique strongly simplifies the general De Giorgi iteration method in combination with Sobolev embedding to prove $L^{\infty}$-bounds. We also point out that this $\delta$-decomposition method is applicable for general nonlocal operators and allows to get explicit constants for the boundedness.
In Section~\ref{section-Equic}, we then prove the local equicontinuity result stated in Theorem~\ref{main-theorem-introduction}(iii). A natural strategy of proving this result is to first obtain a locally uniform estimate for the difference in \eqref{eq:difference-loglap-frac} and then to apply regularity estimate from \cite{RefMA} for weakly singular integral operators which applies, in particular, to the logarithmic Laplacian $\loglap$. However, we are not able to obtain uniform estimates for the difference in \eqref{eq:difference-loglap-frac}. Therefore we first prove uniform bounds related to an auxiliary integral operator instead (see Lemma~\ref{interior-basis2} below) and then complete the proof by a contradiction argument. 
In Section~\ref{sec:unif-bound-decay}, we prove, assuming a uniform exterior sphere condition for $\Omega$, a uniform decay property for the set of eigenfunctions $\{\phi_{k,s}\::\: s \in (0,\frac{1}{4}]\}$. Combining this uniform decay property with the local equicontinuity proved in Section~\ref{section-Equic}, the relative compactness in $C_0(\Omega)$ then follows, as claimed in Theorem~\ref{main-theorem-introduction}(iv).
In Section~\ref{sec:proof-main-theorems}, we finally complete the proof of the main results stated here in the introduction.
\medskip

\textbf{Notation.}
We let $\omega_{N-1}= \frac{2\pi^{\frac{N}{2}}}{\Gamma(\frac{N}{2})} = \frac{2}{C_N}$ denote the measure of the unit sphere in $\R^N$. For a set $A \subset\R^N$ and $x \in \R^N$, we define $\delta_A(x):=\dist(x,A^c)$ with $A^c=\R^N\setminus A$ and, if $A$ is measurable, then $|A|$ denotes its Lebesgue measure. Moreover, for given $r>0$, let $B_r(A):=\{x\in \R^N\;:\; \dist(x,A)<r\}$, and let $B_r(x):=B_r(\{x\})$ denote the ball of radius $r$ with $x$ as its center. If $x=0$ we also write $B_r$ instead of $B_r(0)$.\\
For $A\subset \R^N$ and $u:A\to \R$ we denote $u^+:=\max\{u,0\}$ as the positive and $u^-=-\min\{u,0\}$ as the negative part of $u$, so that $u=u^+-u^-$. Moreover, we let
$$
\underset{A}{\osc}\, u := \sup_A u - \inf_A u \quad \in \: [0,\infty]
$$
denote the oscillation of $u$ over $A$. 
If $A$ is open, we denote by $C^k_c(A)$ the space of function $u:\R^N\to \R$ which are $k$-times continuously differentiable and with support compactly contained in $A$.\\

\textbf{Acknowledgements.} This work is supported by DAAD and BMBF (Germany) within the project 57385104. The authors would like to thank Mouhamed Moustapha Fall for helpful discussions.

\section{First order expansion of eigenvalues and \texorpdfstring{$L^2$}{L2}-convergence of eigenfunctions}\label{section2}

In this section, we first collect some preliminary notions and observations. After this, we complete the proof Theorem~\ref{main-theorem-introduction}(i), see Theorem~\ref{lambda-limit-lower-bound} below.
 
For $s \in (0,1)$, we use the fractional Sobolev space $H^s(\R^N)$ defined as
\begin{equation}\label{wsp-rn}
H^{s}(\R^N)=\Bigg\{u\in L^{2}(\R^N)\;:\; \int_{\R^N}\int_{\R^N}\frac{|u(x)-u(y)|^2}{|x-y|^{N+2s}}\ dxdy<\infty\Bigg\},
\end{equation}
with corresponding norm given by
\begin{equation}\label{N1}
\displaystyle\|u\|_{H^s(\R^N)} = \left(\|u\|^{2}_{L^2(\R^N)} + \int_{\R^N}\int_{\R^N}\frac{|u(x)-u(y)|^2}{|x-y|^{N+2s}}\ dxdy\right)^{\frac{1}{2}}. 
\end{equation}
We recall that this norm is induced by the scalar product
\[
(u,v)\mapsto \langle u,v\rangle_{H^s(\R^N)}=\langle u,v\rangle_{L^2(\R^N)}+ \cE_s(u,v),
\]
where
\begin{equation}
  \label{eq:def-e-s-section}
\cE_s(u,v)=\frac{C_{N,s}}{2}\displaystyle\int_{\R^N}\int_{\R^N}\frac{\big(u(x)-u(y)\big)\big(v(x)-v(y)\big)}{|x-y|^{N+2s}}\ dxdy = \int_{\R^N}|\xi|^{2s} \hat u (\xi) \hat v(\xi)\,d\xi 
\end{equation}
for $u,v \in H^s(\R^N)$ and the constant $C_{N,s}$ is given in \eqref{c-n-s-constant}. The following elementary observations involving the asymptotics of $C_{N,s}$ are used frequently in the paper.

\begin{lemma}
  \label{C-N-s-asymptotics}
  With $C_N = \pi^{-\frac{N}{2}}\Gamma(\frac{N}{2})= \frac{2}{\omega_{N-1}}$ and $\rho_{N} = 2\log2 + \psi(\frac{N}{2})-\gamma$ as defined in the introduction, we have 
  \begin{equation}
    \label{eq:key-est}
  \frac{C_{N,s}}{sC_N} = \frac{\omega_{N-1}C_{N,s}}{2s}= 1 + s\rho_N   + o(s) \qquad \text{as $s \to 0^+$.}
\end{equation}
Consequently, there exists a constant $D_N>0$ with
\begin{equation}
  \label{eq:C-N-D-N-ineq}
\Bigl|1-\frac{C_{N,s}}{C_N s}\Bigr| \le sD_N \quad \text{and therefore}\quad 
\Bigl|C_N-\frac{C_{N,s}}{s}\Bigr| \le sC_N D_N \qquad  \text{for $s \in (0,\frac{1}{4}]$}.
\end{equation}
\end{lemma}

\begin{proof}
  The function
  $$
  s \mapsto \tau(s):= \frac{C_{N,s}}{sC_N} = 4^{s}\frac{\Gamma(\frac{N}{2}+s)}{\Gamma(\frac{N}{2})\Gamma(1-s)}
  $$
  is of class $C^1$ on $[0,1)$ and satisfies $\tau(0)=1$ and $\tau'(0)= \rho_{N}$. Hence \eqref{eq:key-est} follows, and \eqref{eq:C-N-D-N-ineq} is an immediate consequence of \eqref{eq:key-est} and the fact that the function $s \mapsto C_{N,s}$ is continuous.  
\end{proof}

In the remainder of this paper, we assume that $\Omega \subset \R^N$ is an open bounded subset with Lipschitz boundary. For matters of convenience, we identify, for $p \in [1,\infty]$, the space $L^p(\Omega)$ with the space of functions $u \in L^p(\R^N)$ satisfying $u \equiv 0$ on $\Omega^c$.

For $s \in (0,1)$, we then consider the subspace $\cH^{s}_{0}(\Omega) \subset H^{s}(\R^N)$ as defined in \eqref{wsp-d-intro}.
Due to the boundedness of $\Omega$, we have 
\begin{equation}\label{s-first-eigenvalue}
\lambda_{1,s}(\Omega):=\inf_{\substack{u\in \cH^s_0(\Omega) \\ u\neq 0}} \frac{\cE_s(u,u)}{\|u\|_{L^2(\R^N)}^2}>0
\end{equation}
so we can equip the Hilbert space $\cH^s_0(\Omega)$ with the scalar product $\cE_s$ and induced norm 
$$u \mapsto \|u\|_{\cH^{s}_{0}(\Omega)} := \cE_s(u,u)^\frac{1}{2}.$$

Moreover, $\cH^s_0(\Omega)$ is compactly embedded in $L^2(\Omega)$, $C^2_c(\Omega)$ is dense in $\cH^s_0(\Omega)$, and we have
$$
\cE_s(u,v)=\int_{\R^N}u(x)(-\Delta)^s v(x)\ dx \qquad \text{for all $u\in H^s(\R^N)$ and $v\in C^2_c(\R^N)$,}
$$
see \cite{RefV0}. We now set up the corresponding framework of problem~\eqref{eq-eigenvalue-loglap} for the logarithmic Laplacian. We let as in the introduction, see \eqref{scalar-prod-L-intro}, \eqref{log-space-intro},
\begin{equation}\label{log-space}
\cH^0_0(\Omega):=\Bigg\{u\in L^2(\R^N)\;:\; u\equiv 0 \text{ on  $\Omega^c$ and }\iint_{\substack{{x,y\in \R^N}\\|x-y|< 1}}\frac{(u(x)-u(y))^2}{|x-y|^N}dxdy<+\infty \Bigg\}. 
\end{equation}
Then the map
\begin{equation}\label{scalar-prod-L}
(u,v)\mapsto \langle u,v \rangle_{\smhoo} :=  \frac{C_N}{2}\iint_{\substack{{x,y\in \R^N}\\|x-y|< 1}}\frac{(u(x)-u(y))(v(x)-v(y))}{|x-y|^N}\ dxdy,
\end{equation}
is a scalar product on $\cH^0_0(\Omega)$ by \cite[Lemma 2.7]{RefM1}, and the space $\cH^0_0(\Omega)$ is a Hilbert space. Here, $C_N=\pi^{-N/2}\Gamma(\frac{N}{2})=\frac{2}{\omega_{N-1}}$ is as in the introduction. We denote the induced norm by $\|\cdot \|_{\smhoo}$.
Moreover, by \cite[Theorem 2.1]{RefE}), 
\begin{equation}\label{compact-em}
\text{ the  embedding  $\cH^0_0(\Omega) \hookrightarrow L^2(\Omega)$ is compact,}
\end{equation}
and the space $C^2_c(\Omega)$ is dense in $\cH^0_0(\Omega)$ by \cite[Theorem 3.1]{RefT1}.

\begin{remark}
  \label{remark-no-hilbert-space-convergence}
  We stress that, despite the similarities noted above, $\cH^0_0(\Omega)$
  should {\em not} be considered as a limit of the Hilbert spaces $\cH^s_0(\Omega)$ as $s \to 0^+$. In particular, it is not the limit in the sense of \cite{KS03}. Instead, the space $\cH_0^0(\Omega)$ arises naturally when considering a first oder expansion of $\langle  \cdot ,\cdot \rangle_{H^s(\R^N)}$, cf. Lemma~\ref{quadratic-form-est} below.         
\end{remark}

Next we note that, setting 
\begin{equation}\label{EL-diff-sign}
\cE_{0}(u,v) = \langle u,v \rangle_{\smhoo}-C_N\iint_{\substack{{x,y\in \R^N}\\|x-y|\ge 1}}\frac{u(x)v(y)}{|x-y|^N}\ dxdy + \rho_N\int_{\R^N}uv\ dx 
\end{equation}
with $\rho_{N} = 2\log2 + \psi(\frac{N}{2})-\gamma$ as in the introduction, we have 
\[
\cE_0(u,v)=\int_{\Omega} u(x)\loglap v(x)\ dx \quad\text{for $u\in \cH^0_0(\Omega)$ and $v\in C^{1}_c(\Omega)$,}
\]
see \cite{RefT1}. In order to get a convenient parameter-dependent notation for the remainder of this section, we now put
$$
L^s=(-\Delta)^s\quad \text{for $s \in (0,1)$}\qquad \text{and}\qquad L^0=\loglap.
$$
Then, for $s \in [0,1)$, we call $\lambda\in \R$ a Dirichlet-eigenvalue of $L^s$ in $\Omega$ with corresponding eigenfunction $u\in \cH^s_0(\Omega)\setminus \{0\}$ if 
\begin{equation}\label{eigen-prob}
\left\{\begin{aligned} L^su&=\lambda u && \text{in $\Omega$}\\
u&=0 &&\text{in $\Omega^c$,} \end{aligned}\right.
\end{equation}
holds in weak sense, i.e., if
\[
\cE_s(u,\psi) = \lambda\int_{\Omega}u\psi\ dx \qquad \text{ for all $\psi\in \cH^s_0(\Omega)$.}
\]
In the following Proposition we collect the known properties on the eigenvalues and eigenfunctions of the fractional Laplacian and the logarithmic Laplacian, see e.g. \cite[Prosition 3.1]{RefG} and the references in there for the fractional Laplacian and \cite[Theorem 3.4]{RefT1}) for the logarithmic Laplacian. 

\begin{prop}\label{eigen-frac}
Let $\Omega\subset \R^N$ be an open bounded set with Lipschitz boundary, and let $s\in [0,1)$. Then the following holds: 
\begin{itemize}
\item[(a)] The eigenvalues of problem \eqref{eigen-prob} consist of a sequence $\{\lambda_{k,s}(\Omega)\}_{k\in\N}$~ with\\
$0<\lambda_{1,s}(\Omega)< \lambda_{2,s}(\Omega)\le \cdots\le \lambda_{k,s}(\Omega)\le \lambda_{k+1,s}(\Omega)\le \cdots ~~ \text{and} ~~\lim\limits_{k\rightarrow \infty} \lambda_{k,s}(\Omega) = +\infty.$
\item[(b)]The sequence $\{ \phi_{k,s}\}_{k\in\N}$ of eigenfunctions corresponding to eigenvalues $\lambda_{k,s}(\Omega)$ forms a complete orthonormal basis of $L^2(\Omega)$ and an orthogonal system of $\cH_0^s(\Omega)$.
\item[(c)] For any $k\in\N$, the eigenvalue $\lambda_{k,s}(\Omega)$ is characterized as 
 $$
  \lambda_{k,s}(\Omega) = \min\left\{ \cE_s(u,u): u\in \mathbb{P}_{k,s}(\Omega)~~~\text{ and }~~~\|u\|_{L^2(\Omega)}=1\right\},
$$
where $\mathbb{P}_{1,s}(\Omega) = \cH^s_0(\Omega)$ and
$$
\mathbb{P}_{k,s}(\Omega) = \left\{u\in \cH^s_0(\Omega) :\cE_s(u,\phi_{j,s}) = 0\; \text{for $j=1,\cdots,k-1$}\right\}\qquad \text{for $k>1$.}
$$
\item[(d)] The first eigenvalue $\lambda_{1,s}(\Omega)$ is simple and the corresponding eigenfunction $\phi_{1,s}$ does not change its sign in $\Omega$ and can be chosen to be strictly positive in $\Omega$.
\end{itemize}
\end{prop}

\begin{remark}
\label{eigen-frac-remark}  
(i) The characterization in Proposition~\ref{eigen-frac}(c) implies that $\lambda_{1,s}(\Omega)$, as defined in \eqref{s-first-eigenvalue}, is indeed the first Dirichlet eigenvalue of $(-\Delta)^s$ on $\Omega$, so the notation is consistent.\\ 

(ii) We emphasize that in the case $s=0$ the eigenvalues $\lambda_{k,0}$ and corresponding eigenfunctions $\phi_{k,0}$ for $k\in\N$ are also denoted by $\lambda_{k,L}$ and $\phi_{k,L}$ resp. as in the introduction for consistency.\\
 
(iii) By the Courant-Fischer minimax principle and due to the density of $C^2_c(\Omega)$ in $\cH^s_0(\Omega)$, the eigenvalues $\lambda_{k,s}$, $s \in [0,1)$, $k \in \N$ can be characterized equivalently as 
\begin{equation}\label{char-eigen}
\lambda_{k,s}(\Omega)= \inf_{\substack{V\subset \cH^s_0(\Omega)\\ \dim V = k}}\:\max_{\substack{v\in V\setminus\{0\} \\ \|v\|_{L^2(\Omega)}=1}} \cE_{s}(v,v) =\inf_{\substack{V\subset C^2_c(\Omega)\\ \dim V = k}}\: \max_{\substack{v\in V\setminus\{0\} \\ \|v\|_{L^2(\Omega)}=1}} \cE_{s}(v,v).
\end{equation}
This fact will be used in the sequel.
\end{remark}

Next, we need the following elementary estimates.

\begin{lemma}
\label{elementary-est}
For $s \in (0,1)$ and $r>0$ we have 
\begin{equation}
\label{elementary-est-1}
\Bigl|\frac{r^{2s}-1}{s}\Bigr| \le 2  \Bigl( |\ln r| 1_{(0,1]}(r) + 
1_{(1,\infty)}(r)r^4 \Bigr)   
\end{equation}
and 
\begin{equation}
\label{elementary-est-2}
\Bigl|\frac{r^{2s}-1}{s} - 2 \log r\Bigr| \le 4s \Bigl( \ln^2(r) 1_{(0,1]}(r) + 
1_{(1,\infty)}(r)r^4 \Bigr) 
\end{equation}
\end{lemma}

\begin{proof}
Fix $r>0$ and let $h_r(s)=r^{2s}$, $r>0$. Then we have $h_r'(\tau)= 2 r^{2\tau} \ln r$ and $h_r''(\tau)= 4 r^{2\tau} \ln^2 (r) $ for $\tau>0$. Consequently,
$$
\Bigl|\frac{r^{2s}-1}{s}\Bigr|= \frac{2 |\ln r|}{s}\int_0^s r^{2\tau}\,d\tau
\le 2 | \ln r | r^{2s} \le  2  \Bigl( |\ln r| 1_{(0,1]}(r) + 
1_{(1,\infty)}(r)r^4 \Bigr),
$$
where in the last step we used that $r^{2s} \le 1$ for $r \le 1$ and, since $s<1$,
$$
r^{2s} \ln r  \leq r^{2s+1}  \le r^4 \quad \text{for $r>1$.}$$
Hence \eqref{elementary-est-1} is true. Moreover, by Taylor expansion,
$$
h_r(s)= 1+s h_r'(0)+ \int_{0}^s h''_r(\tau)(s-\tau) d\tau = 1 + 2s \ln r + 4 \ln^2 r \int_{0}^s r^{2\tau}(s-\tau)d\tau
$$
and therefore 
$$
\Bigl|\frac{r^{2s}-1}{s} - 2 \log r\Bigr| \le \frac{4 \ln^2 (r)}{s}\Bigl|\int_{0}^s r^{2\tau}(s-\tau)d\tau\Bigr| \le 4 s r^{2s}\ln^2(r) .
$$
Hence \eqref{elementary-est-2} follows since for $r \in (0,1]$ we have $r^{2s} \le 1$ and, since $s<1$,
$$
r^{2s} \ln^2 r  = r^{2s+2}  \le r^4 \quad \text{for $r>1$.}  
$$
\end{proof}

\begin{lemma}
\label{quadratic-form-est}
For every $u \in C^2_c(\Omega)$ and $s \in (0,1)$ we have 
\begin{equation}
  \label{eq:quadratic-form-est-eq1}
\Bigl|\cE_s(u,u) - \|u\|_{L^2(\R^N)}^2 \Bigr| \le 2 s\Bigl( \kappa_{N} \|u\|_{L^1(\R^N)}^2 + \|\Delta u\|_{L^2(\R^N)}^2 \Bigr)  
\end{equation}
and 
\begin{equation}
  \label{eq:quadratic-form-est-eq2}
\Bigl|\cE_s(u,u) - \|u\|_{L^2(\R^N)}^2 - s \cE_0(u,u)\Bigr| \le 4 s^2 \Bigl( \kappa_{N} \|u\|_{L^1(\R^N)}^2 + \|\Delta u\|_{L^2(\R^N)}^2 \Bigr)
\end{equation}
with $\kappa_{N} = (2\pi)^{-N} \int_{B_1(0)} \ln^2 |\xi| \,d \xi.$
\end{lemma}

\begin{proof}
Let $u \in C^2_c(\Omega)$ and $s \in (0,1)$. By \eqref{eq:def-e-s-section} and~\eqref{elementary-est-1}, we have 
\begin{align*}
 \Bigl|\cE_s(u,u) - \|u\|_{L^2(\R^N)}^2\Bigr| &\le 
\int_{\R^N}\bigl| |\xi|^{2s}-1 \bigr|\, |\hat u(\xi)|^2\,d\xi\\
&\le 2s \Bigl( \int_{B_1(0)}|\ln |\xi|| |\hat u(\xi)|^2\,d \xi + \int_{\R^N \setminus B_1}|\xi|^4 |\hat u(\xi)|^2\,d \xi \Bigr) \\
&\le 2s \Bigl( \|\hat u\|_{L^\infty(\R^N)}^2 \int_{B_1(0)} |\ln |\xi|| \,d \xi + \|\Delta u\|_{L^2(\R^N)}^2 \Bigr)\\
&\le 2s \Bigl( (2\pi)^{-N} \|u\|_{L^1(\R^N)}^2 \int_{B_1(0)} \ln^2 |\xi| \,d \xi + \|\Delta u\|_{L^2(\R^N)}^2\Bigr).
\end{align*}
Thus  \eqref{eq:quadratic-form-est-eq1} follows. Moreover, by~\eqref{elementary-est-2} we have  
\begin{align*}
 \Bigl|\cE_s(u,u) - \|u\|_{L^2(\R^N)}^2 - s \cE_0(u,u)\Bigr| &\le 
\int_{\R^N}\bigl| |\xi|^{2s}-1 - 2 s \log |\xi|\bigr|\, |\hat u(\xi)|^2\,d\xi\\
&\le 4s^2 \Bigl( \int_{B_1(0)}\ln^2 |\xi| |\hat u(\xi)|^2\,d \xi  
+ \int_{\R^N \setminus B_1}|\xi|^4 |\hat u(\xi)|^2\,d \xi \Bigr) \\
&\le 4s^2 \Bigl( \|\hat u\|_{L^\infty}^2 \int_{B_1(0)} \ln^2 |\xi| \,d \xi  
+ \|\Delta u\|_{L^2(\R^N)}^2 \Bigr)\\
&\le 4s^2 \Bigl( (2\pi)^{-N} \|u\|_{L^1(\R^N)}^2 \int_{B_1(0)} \ln^2 |\xi| \,d \xi + 
\|\Delta u\|_{L^2(\R^N)}^2\Bigr)
\end{align*}
Hence \eqref{eq:quadratic-form-est-eq2} follows. 
\end{proof}

\begin{lemma}
\label{upper-est-lambda_k}
For all $k \in \N$ we have 
\begin{equation}
  \label{eq:lim-inf-limsup-prelim}
\lambda_{1,0}(\Omega) \,\le \,\liminf_{s \to 0^+} \frac{\lambda_{k,s}(\Omega)-1}{s} \, \le \, 
\limsup_{s \to 0^+} \frac{\lambda_{k,s}(\Omega)-1}{s} \, \le \, \lambda_{k,0}(\Omega)\end{equation}
and 
\begin{equation}
\label{uniform-upper-bound}
\lambda_{k,s}(\Omega) \le 1 +sC \qquad \text{for all $s \in (0,1)$}
\end{equation}
with a constant $C= C(N,\Omega,k)>0$.
\end{lemma}

\begin{proof}
We fix a subspace $V\subset C^2_c(\Omega)$ of dimension $k$ and let $S_V:=\{u\in V\;:\; \|u\|_{L^2(\Omega)}=1\}$. Using \eqref{char-eigen} and \eqref{eq:quadratic-form-est-eq1}, we find that, for $s \in (0,1)$,  
\begin{equation}
  \label{eq:upper-est-lambda_k-first-eq}
\frac{\lambda_{k,s}(\Omega)- 1}{s} \le \max\limits_{u\in S_V} \frac{\cE_s(u,u) - 1}{s} \le C
\end{equation}
with 
$$
C = C(N,\Omega,k) = 2\max\limits_{u\in S_V}\Bigl(\kappa_{N} \|u\|_{L^1(\R^N)}^2 + \|\Delta u\|_{L^2(\R^N)}^2 \Bigr).
$$
Hence \eqref{uniform-upper-bound} holds. Moreover, setting $\cR_s(u)= \frac{\cE_s(u,u) - 1}{s}- \cE_0(u,u)$ for $u \in C^2_c(\Omega)$, we deduce from \eqref{eq:upper-est-lambda_k-first-eq} that 
$$ 
\frac{\lambda_{k,s}(\Omega)- 1}{s} \le \max\limits_{u\in S_V}\cE_0(u,u) 
+ \max\limits_{u\in S_V}|\cR_s(u)| 
$$
while, by Lemma~\ref{quadratic-form-est}, 
$$
|\cR_s(u)| \le 4 s \Bigl(\kappa_{N}  \|u\|_{L^1(\R^N)}^2 + \|\Delta u\|_{L^2(\R^N)}^2\Bigr) \to 0 \quad \text{as $s \to 0^+$ uniformly in $u \in S_V$.}
$$
Consequently, 
$$
\limsup_{s \to 0^+} \frac{\lambda_{k,s}(\Omega)- 1}{s} \le \max\limits_{u\in S_V}\cE_0(u,u).
$$
Since $V$ was chosen arbitrarily, the characterization of the Dirichlet eigenvalues of the logarithmic Laplacian given in \eqref{char-eigen} with $s=0$ implies that 
\begin{equation}\label{derivative-k}
\limsup_{s\rightarrow 0^+} \frac{\lambda_{k,s}(\Omega)- 1}{s} \le \inf_{\substack{V\subset C^2_c(\Omega) \\ \dim(V)=k}}\max_{\substack{u\in V\\ \|u\|_{L^2(\Omega)}=1}}\cE_L(u,u) = \lambda_{k,0}(\Omega), 
\end{equation}
In particular, the last inequality in \eqref{eq:lim-inf-limsup-prelim} holds. Moreover, since $\lambda_{k,s}(\Omega) \ge \lambda_{1,s}(\Omega)$ for every $k \in \N$ and 
$$
\lim \limits_{s \to 0^+} \frac{\lambda_{1,s}(\Omega)-1}{s}= \lambda_{1,0}(\Omega)
$$ 
by \cite[Theorem 1.5]{RefT1}, the first inequality in \eqref{eq:lim-inf-limsup-prelim} also follows. 
\end{proof}

\begin{cor}
\label{rem-limit-one}
For all $k \in \N$ we have $\lim \limits_{s \to 0^+} \lambda_{k,s}(\Omega)= 1$.
\end{cor}

\begin{proof}
This immediately follows from \eqref{eq:lim-inf-limsup-prelim}.
\end{proof}

\begin{lemma}\label{prop-b}
Let $k \in \N$, $s_0 \in (0,1)$, and let, for $s \in(0,s_0)$, $\phi_{k,s} \in \cH^s_0(\Omega)$ denote an $L^2$-normalized eigenfunction of $(-\Delta)^s$ in $\Omega$. Then the set 
$$
\{ \phi_{k,s}\::\: s \in (0,s_0)\}
$$
is uniformly bounded in $\cH^0_0(\Omega)$ and therefore relatively compact in $L^2(\Omega)$.
 \end{lemma}
 
\begin{proof}
By \eqref{uniform-upper-bound}, there exists a constant $C=C(N,\Omega,k)>0$ with the property that 
\begin{align}
C&\ge \frac{\lambda_{k,s}(\Omega)- 1}{s}=\frac{\cE_{s}(\phi_{k,s},\phi_{k,s})- 1}{s}=   \frac{C_{N,s}}{2s}\int_{\R^N}\int_{\R^N}\frac{|\phi_{k,s}(x)-\phi_{k,s}(y)|^{2}}{|x-y|^{N+2s}}dxdy-\frac{1}{s}\notag\\
&  = \frac{C_{N,s}}{2s}\iint_{|x-y|< 1}\frac{|\phi_{k,s}(x)-\phi_{k,s}(y)|^2}{|x-y|^{N+2s}}dxdy- \frac{C_{N,s}}{s}\iint_{|x-y|\ge 1}\frac{\phi_{k,s}(x)\phi_{k,s}(y)}{|x-y|^{N+2s}}dxdy+f_N(s),\label{prob-b:1}
\end{align}
where, due to the $L^2$-normalization of $\phi_{k,s}$,   
\begin{equation}\label{prob-b:2}
f_N(s) := \frac{1}{s}\Big(C_{N,s}\int_{\Omega}|\phi_{k,s}(x)|^2 \int_{\R^N\setminus B_1(x)}\frac{1}{|x-y|^{N+2s}}dy\,dx-1\Big) = \frac{1}{s}\left(\frac{C_{N,s}\omega_{N-1}}{2s}-1\right).
\end{equation}
Therefore, using the definition of $\|\cdot\|_{\cH^0_0(\Omega)}$, we deduce that 
\begin{equation}
C\ge  \frac{C_{N,s}}{sC_N}\|\phi_{k,s}\|_{\cH_0^0(\Omega)}^2- \frac{C_{N,s}}{s}\iint_{|x-y|\ge 1}\frac{|\phi_{k,s}(x)\phi_{k,s}(y)|}{|x-y|^{N+2s}}dxdy  +f_N(s),\label{prob-b:3}
\end{equation}
where, by H\"older's inequality,
\[
\begin{split}
\iint_{|x-y|\ge 1}\frac{|\phi_{k,s}(x)\phi_{k,s}(y)|}{|x-y|^{N+2s}}dxdy\le \int_{\Omega}\int_{\Omega\cap\{|x-y|\ge 1\}}\frac{|\phi_{k,s}(x)|^2}{|x-y|^{N}}dydx\le |\Omega|\|\phi_{k,s}\|_{L^2(\Omega)} = |\Omega|,
\end{split}
\]
using again the $L^2$-normalization. Combining this with \eqref{prob-b:3}, we find that 
\[
\|\phi_{k,s}\|_{\cH_0^0(\Omega)}^2\le \frac{sC_{N}}{C_{N,s}}\Big(C +|\Omega|  - f_N(s)\Big).
\]
Since moreover $\frac{sC_N }{C_{N,s}}\to 1$ and $f_{N}(s) \to \rho_N$ as $s\to 0^+$ by Lemma~\ref{C-N-s-asymptotics}, we conclude that there exists a constant $K=K(N,k,\Omega)>0$ and $s_1\in(0,1)$ such that
\[
\|\phi_{k,s}\|_{\cH_0^0(\Omega)} \le K ~~~~~~\text{ for all $s\in(0,s_1)$.}
\]
Consequently, the set $\{ \phi_{k,s}\::\: s \in (0,s_1)\}$ is uniformly bounded in $\cH_0^0(\Omega)$ and thus relatively compact in $L^2(\Omega)$ by \eqref{compact-em}. Hence the claim follows for $s_0 \le s_1$.

If $s_0 \in (s_1,1)$, we can use the fact that by \eqref{uniform-upper-bound} we have, for $s \in [s_1,s_0]$,  
\begin{align*}
  1+C &\ge \lambda_{k,s}(\Omega) = \cE_s(\phi_{k,s},\phi_{k,s})= \frac{C_{N,s}}{2}\int_{\R^N}\int_{\R^N}\frac{|\phi_{k,s}(x)-\phi_{k,s}(y)|^{2}}{|x-y|^{N+2s}}dxdy\\ 
 &\ge \frac{C_{N,s}}{2}\iint_{|x-y|\le 1}\frac{|\phi_{k,s}(x)-\phi_{k,s}(y)|^{2}}{|x-y|^{N}}dxdy =\frac{C_{N,s}}{C_N}\|\phi_{k,s}\|_{\cH_0^0(\Omega)}^2
\end{align*}
with a constant $C= C(N,\Omega,k)>0$ and hence
$$
\sup_{s \in [s_1,s_0]}\|\phi_{k,s}\|_{\cH_0^0(\Omega)}^2 \le C_N (1+C) \sup_{s \in [s_1,s_0]}\frac{1}{C_{N,s}}<\infty.
$$
We thus conclude that the set $\{ \phi_{k,s}\::\: s \in (0,s_0)\}$ is uniformly bounded in $\cH_0^0(\Omega)$ and thus relatively compact in $L^2(\Omega)$ by \eqref{compact-em}, as claimed.
\end{proof}

We finish this section with the the following theorem which, in particular, completes the proof of Theorem~\ref{rem-main-thm-intro}(i).

\begin{thm}
\label{lambda-limit-lower-bound}
For every $k \in \N$ we have 
\begin{equation}\label{derivative-k2-0}
 \lim_{s \to 0^+} \frac{\lambda_{k,s}(\Omega)- 1}{s}= \lambda_{k,0}(\Omega).
\end{equation}
Moreover, if $(s_n)_n \subset (0,1)$ is a sequence such that $\lim\limits_{n\to \infty}s_n=0$ and $\phi_{k,s_n}$ is an $L^2$-normalized Dirichlet eigenfunction of $(-\Delta)^s$ corresponding to the eigenvalue $\lambda_{k,s}(\Omega)$, then, after passing to a subsequence,
\[
\phi_{k,s} \to \phi_{k,0} ~~~~~\text{ in $L^2(\Omega)\ $ as  $n\to \infty$,}
\]
where $\phi_{k,0}$ is an $L^2$-normalized Dirichlet eigenfunction of the logarithmic Laplacian corresponding to $\lambda_{k,0}(\Omega)$.
\end{thm}

\begin{proof}
To establish \eqref{derivative-k2-0}, it suffices, in view of \eqref{eq:lim-inf-limsup-prelim}, to consider an arbitrary sequence $(s_n)_n\subset(0,1)$ with $\lim\limits_{n\to \infty}s_n=0$, and to show that, after passing to a subsequence, 
\begin{equation}\label{derivative-k2}
\lim_{n \to \infty} \frac{\lambda_{k,s_n}(\Omega)- 1}{s}= \lambda_{k,0}(\Omega) \qquad \text{for $k \in \N$.}
\end{equation}
Let $\{\phi_{k,s_n}\::\: k \in \N\}$ be an orthonormal system of eigenfunctions corresponding to the Dirichlet eigenvalue $\lambda_{k,s_n}(\Omega)$ of $(-\Delta)^{s_n}$. By Lemma \ref{prop-b}, it follows that, for every $k \in \N$, the sequence of functions 
$\phi_{k,s_n}$, $n\in \N$ is bounded in $\cH^0_0(\Omega)$ and relatively compact in $L^2(\Omega)$. Consequently, we may pass to a subsequence such that, for every $k \in \N$, 
\begin{equation}
   \phi_{k,s_n} \rightharpoonup \phi_{k,0}\:\text{ weakly  in  $\cH^0_0(\Omega)$} \;\: \text{and}\;\:  \phi_{k,s_n} \rightarrow \phi_{k,0}\: \text{ strongly  in $L^2(\Omega)$}\quad \text{as $n \to \infty$.}\label{conv-of-phi}
\end{equation}
Moreover, by \eqref{eq:lim-inf-limsup-prelim} we may, after passing again to a subsequence if necessary, assume that, for every $k \in \N$, 
\begin{equation}
\displaystyle \frac{\lambda_{k,s_n}(\Omega)-1}{s_n}\: \rightarrow \: \lambda^{\star}_k\in \Big[\lambda_{1,0}(\Omega),\lambda_{k,0}(\Omega)\Big] \qquad \text{as $n \to \infty$.} \label{conv-of-lamda}
\end{equation}
To prove \eqref{derivative-k2}, it now suffices to show that 
\begin{equation}\label{derivative-k2-sufficient}
\lambda_{k,0}(\Omega)= \lambda^{\star}_k \qquad \text{for every $k \in \N$.}
\end{equation}
It follows from \eqref{conv-of-phi} that
\begin{equation}
  \label{eq:orthonormal-limit-system}
\|\phi_{k,0}\|_{L^2(\Omega)} = 1 \quad \text{and} 
\quad \langle\phi_{k,0},\phi_{\ell,0}\rangle_{L^2(\Omega)}= 0\qquad \text{for $k, \ell \in \N$, $\ell \not = k$.}
\end{equation}
Moreover, for $w\in C^2_c(\Omega)$ and $n \in \N$ we have
\begin{equation}\label{test-phi-k-s}
 \cE_{s_n}(\phi_{k,s_n},w) = \lambda_{k,s_n}(\Omega)\langle\phi_{k,s_n},w\rangle_{L^2(\Omega)}
\end{equation} 
and therefore, by \cite[Theorem 1.1 (i)]{RefT1},
\begin{align*}
&\lim_{n\rightarrow \infty}\frac{\lambda_{k,s_n}(\Omega)-1}{s_n}\langle \phi_{k,s_n},w\rangle_{L^2(\Omega)}= \lim_{n\rightarrow \infty}\frac{1}{s_n}\Big( \cE_s(\phi_{k,s_n},w)-\langle\phi_{k,s_n},w\rangle_{L^2(\Omega)}\Big) \\
&=\lim_{n\rightarrow\infty}\Big\langle\phi_{k,s_n},\frac{(-\Delta)^{s_n}w-w}{s_n}\Big\rangle_{L^2(\Omega)} =\langle \phi_{k,0},\loglap w\rangle_{L^2(\Omega)} = \cE_{L}(\phi_{k,0},w).
\end{align*}
Since moreover $\langle\phi_{k,s_n},w\rangle_{L^2(\Omega)}\to \langle\phi_{k,0},w\rangle_{L^2(\Omega)}$  for $n\to \infty$, it follows from \eqref{conv-of-lamda} that 
\begin{equation}\label{phi-s0}
\cE_{L}(\phi_{k,0},w) = \lambda^{\star}_k \langle \phi_{k,0}, w \rangle_{L^2(\Omega)} ~~~~\text{ for all $w\in C^2_c(\Omega)$.}
\end{equation}
Thus $\phi_{k,0}$ is a Dirichlet eigenfunction of the logarithmic Laplacian $L_{0}$ corresponding to $\lambda^{\star}_k$.

Next, for fixed $k \in \N$, we consider $E_{k,0}:=\text{span}\{\phi_{1,0},\phi_{2,0},\cdots,\phi_{k,0}\}$, which is a $k$-dimensional subspace of $\cH^0_0(\Omega)$  by \eqref{eq:orthonormal-limit-system}. Since 
\[
\lambda^{\star}_1\le \lambda^{\star}_2\leq \ldots \leq \lambda^{\star}_k
\]
as a consequence of \eqref{conv-of-lamda} and since $\lambda_{i,s_n}\leq \lambda_{j,s_n}$ for $1\leq i\leq j\leq k$, $n\in\N$, we have the following estimate for every $w =\sum\limits_{i=1}^k\alpha_{i}\phi_{i,0} \in E_{k,0}$ with $\alpha_{1},\cdots,\alpha_{k}\in \R$:
\begin{align}
  \cE_0(w,w) &=  \sum_{i,j=1}^k\alpha_{i}\alpha_{j}\cE_0(\phi_{i,0},\phi_{j,0})=
               \sum_{i,j=1}^k\alpha_{i}\alpha_{j}\lambda_{i}^* \langle \phi_{i,0},\phi_{j,0} \rangle_{L^2(\Omega)}\\
& = \sum_{i=1}^k \alpha_{i}^2 \lambda_{i}^* \|\phi_{i,0}\|_{L^2(\Omega)}^2 \le \lambda_{k}^* \sum_{i=1}^k \alpha_{i}^2 = \lambda_{k}^* \|w\|_{L^2(\Omega)}^2.   
\label{sum}
\end{align}
The characterization in \eqref{char-eigen} now yields that 
\[
\lambda_{k,0}(\Omega) \le \max_{\substack{w\in E_{k,0} \\ \|w\|_{L^2(\Omega)}=1}}\cE_0(w,w) \le \lambda_{k}^{\star}.
\]
Since also $\lambda_{k}^{\star} \le \lambda_{k,0}(\Omega)$ by \eqref{conv-of-lamda}, \eqref{derivative-k2-sufficient} follows. We thus conclude that \eqref{derivative-k2} holds. Moreover, the second statement of the theorem also follows a posteriori from the equality $\lambda_{k}^{\star} = \lambda_{k,0}(\Omega)$, since we have already seen that $\phi_{k,s_n} \to \phi_{k,0}$ in $L^2(\Omega)$, where $\phi_{k,0}$ is a Dirichlet eigenfunction of the logarithmic Laplacian $L_{0}$ corresponding to the eigenvalue $\lambda^{\star}_k$. The proof is thus finished.
\end{proof}
 
 \section{Uniform \texorpdfstring{$L^\infty$}{L-infty}-bounds on eigenfunctions}\label{section-regularity-of-LD}
Through the remainder of this paper, we fix $k\in \N$, and we consider, for $s\in(0,\frac{1}{4}]$, eigenfunctions $\phi_s:=\phi_{k,s}$ of $(-\Delta)^s$ in $\Omega$ corresponding to $\lambda_{s}:=\lambda_{k,s}$. Furthermore, we assume that $\phi_s$ is $L^2$-normalized, that is $\|\phi_s\|_{L^2(\Omega)}=1$ for all $s\in(0,\frac{1}{4}]$. The main result of this section is the following.

\begin{thm}
  \label{uniform-l-infty-section}
There exists a constant $C=C(N,\Omega,k)$ with the property that $\|\phi_s\|_{L^\infty(\Omega)}\le C$ for all $s \in (0,\frac{1}{4}]$.  
\end{thm}

To prove this result, we use a new approach based on a so-called $\delta$-decomposition of nonlocal quadratic forms.\\

 For $\delta>0$ and $u,v \in H^s(\R^N)$, we can write
 \begin{align*}
   \cE_s(u,v)&= \cE_s^\delta(u,v) + \frac{C_{N,s}}{2} \iint_{|x-y|>\delta} \frac{(u(x)-u(y))(v(x)-v(y))}{|x-y|^{N+2s}}\,dxdy\Bigr)\\
   &= \cE_s^\delta(u,v)+   
\kappa_{\delta,s}  \langle u,v \rangle_{L^2(\R^N)} -  \langle k_{\delta,s} * u, v \rangle_{L^2(\R^N)}   
 \end{align*}
 with the $\delta$-dependent quadratic form  
 $$
(u,v) \mapsto \cE_s^\delta(u,v)= \frac{C_{N,s}}{2}\iint_{|x-y|<\delta} \frac{(u(x)-u(y))(v(x)-v(y))}{|x-y|^{N+2s}}\,dxdy,
  $$
  the function $k_{\delta,s} = C_{N,s}1_{\text{\tiny $\R^N \setminus B_{\delta}(0)$}}|\cdot |^{-N-2s}\in L^1(\R^N)$ and the constant
  $$
  \kappa_{\delta,s} = \frac{C_{N,s}\omega_{N-1}\delta^{-2s}}{2s}.
  $$
  In particular, this decomposition is valid if $\Omega \subset \R^N$ is a bounded Lipschitz domain and $u,v \in \cH^{s}_{0}(\Omega)$.

\begin{proof}[Proof of Theorem~\ref{uniform-l-infty-section}]
  Let $\delta\in(0,1)$, $c>0$, and consider the function $w_c=(\phi_s-c)^+: \Omega \to \R$ for $s \in (0,1)$. Then $w_c\in \cH^0_0(\Omega)$ by \cite[Lemma 3.2]{RefJ2}. Moreover, for $x,y \in \R^N$ we have
  \begin{align*}
    &(\phi_s(x)-\phi_s(y))(w_c(x)-w_c(y)) = ([\phi_s(x)-c]-[\phi_s(y)-c])(w_c(x)-w_c(y))\\
&=[\phi_s(x)-c]w_c(x)+[\phi_s(y)-c]w_c(y)-[\phi_s(x)-c]w_c(y) -w_c(x)[\phi_s(y)-c]\\
    &=w_c^2(x)+w_c^2(y)-2 w_c(x)w_c(y)+ [\phi_s(x)-c]^- w_c(y)+ w_c(x)[\phi_s(y)-c]^-\\
    &\ge w_c^2(x)+w_c^2(y)-2 w_c(x)w_c(y) = (w_c(x)-w_c(y))^2,
  \end{align*}
which implies that   
\begin{align}
  &\cE_{s}^\delta(w_c,w_c)= \frac{C_{N,s}}{2}\iint_{|x-y|<\delta}\frac{(w_c(x)-w_c(y))^2}{|x-y|^{N+2s}}\,dxdy   \label{result2:eq1}\\
  &\le   \frac{C_{N,s}}{2}\iint_{|x-y|<\delta} \frac{(\phi_s(x)-\phi_s(y))(w_c(x)-w_c(y))}{|x-y|^{N+2s}}\,dxdy\nonumber\\
  & = \cE_{s}^\delta(\phi_s,w_c) = \cE_s(\phi_s,w_c)-\kappa_{\delta,s} \langle \phi_s,w_c \rangle_{L^2(\Omega)} + \langle k_{\delta,s} * \phi_s\,,w_c \rangle_{L^2(\Omega)} \nonumber\\
&= \bigl(\lambda_s-\kappa_{\delta,s}\bigr) \langle \phi_s,w_c \rangle_{L^2(\Omega)} + \langle k_{\delta,s} * \phi_s\,,w_c \rangle_{L^2(\Omega)}= g_\delta(s)\langle \phi_s,w_c \rangle_{L^2(\Omega)} + \langle k_{\delta,s} * \phi_s\,,w_c \rangle_{L^2(\Omega)}\nonumber
\end{align}
with the function
\begin{equation}
  \label{eq:def-g-delta}
g_\delta:(0,1)\to \R, \qquad g_\delta(s)= \lambda_s -\kappa_{\delta,s}= \lambda_s- \frac{C_{N,s}\omega_{N-1}\delta^{-2s}}{2s}.
\end{equation}
Since $\lambda_s = 1 + \lambda_L s + o(s)$ by Theorem \ref{lambda-limit-lower-bound}, where $\lambda_L=\lambda_{k,0}$ denotes the $k$-the eigenvalue of the logarithmic Laplacian, and 
$$
\frac{C_{N,s}\omega_{N-1}\delta^{-2s}}{2s} = 1 + \bigl(\rho_N + 2 \ln \delta\bigr)s +o(s)\qquad \text{as $s \to 0^+$}
$$
by Lemma~\ref{C-N-s-asymptotics}, we have
$$
g_\delta(s) = \bigl(\lambda_L - \rho_N + 2 \ln \delta\bigr)s+o(s)\qquad \text{as $s \to 0^+$.}
$$
Here the remainder term $o(s)$ depends on $\delta>0$. Nevertheless, we may first fix $\delta \in (0,1)$ sufficiently small such that $\lambda_L - \rho_N + 2 \ln \delta
<-1$, and then we may fix $s_0 \in (0,\frac{1}{4}]$ with the property that 
\begin{equation}
\label{bdd-above-gs}
g_\delta(s)\leq -s \le 0 \quad \text{ for all $s\in (0,s_0]$.}
\end{equation}
Since also $\phi_s(x)w_c(x) \ge c w_c(x) \ge 0$ for $x \in \Omega$, $s \in (0,s_0]$, we deduce from \eqref{result2:eq1} that
\begin{equation}
  \label{result2:eq1-1}
  \cE_{s}^\delta(w_c,w_c)  \le  \int_{\Omega} [k_{\delta,s} * \phi_s - s c] w_c \,dx \le \bigl(\|k_{\delta,s} * \phi_s\|_{L^{\infty}(\Omega)}- s c \bigr) \int_{\Omega} w_c \,dx.
\end{equation}
Here we note that, by H{\"o}lder's (or Young's) inequality, 
$$
\|k_{\delta,s} * \phi_s\|_{L^{\infty}(\Omega)} \leq \|k_{\delta,s}\|_{L^2(\R^N)} \|\phi_s\|_{L^2(\Omega)}= \|k_{\delta,s}\|_{L^2(\R^N)}
$$
with
$$
\|k_{\delta,s}\|_{L^2(\R^N)} = C_{N,s}\Bigg(\ \int_{\R^N\setminus B_{\delta}} |y|^{-2N-4s}\ dy\Bigg)^{1/2}=\frac{C_{N,s}\omega_{N-1}^{\frac{1}{2}}\delta^{-\frac{N}{2}-2s}}{\sqrt{N+4s}}.
$$
Since
$$
\tilde d:= \sup_{s \in (0,s_0]} \frac{\|k_{\delta,s}\|_{L^2(\R^N)}}{s} = \sup_{s \in (0,s_0]}\frac{C_{N,s}\omega_{N-1}^{\frac{1}{2}}\delta^{-\frac{N}{2}-2s}}{s \sqrt{N+4s}}
< \infty,
$$
we deduce from \eqref{result2:eq1-1} that for $c> \tilde d$ and $s \in (0,s_0]$ we have 
\begin{equation*}
0 \le  \cE_{s}^\delta(w_c,w_c)  \le  s \bigl(\tilde d - c\bigr) \int_{\Omega}  w_c \,dx \le 0
\end{equation*}
and therefore $\cE_{s}^\delta(w_c,w_c)=0$. Consequently, $w_c=0$ in $\Omega$ for $s \in (0,s_0]$ by the Poincar\'e type inequality given in \cite[Lemma 2.7]{RefM1} . But then $\phi_s(x)\leq c$ a.e. in $\Omega$, and therefore
$$
\sup_{s \in (0,s_0]}\|\phi_s^+\|_{L^\infty(\Omega)} \le c.
$$
Repeating the above argument for $-\phi_s$ in place of $\phi_s$, we also find that
$\sup \limits_{s \in (0,s_0]}\|\phi_s^-\|_{L^\infty(\Omega)} \le c$ and therefore
\begin{equation}
  \label{eq:uniform-bound-s-0}
\sup \limits_{s \in (0,s_0]}\|\phi_s\|_{L^\infty(\Omega)} \le c.
\end{equation}
It remains to prove that 
\begin{equation}
  \label{eq:uniform-bound-1-third}
\sup \limits_{s \in [s_0,\frac{1}{4}]}\|\phi_s\|_{L^\infty(\Omega)} <\infty.
\end{equation}
To see this, we argue as above, but with different values of $\delta \in (0,1)$ and $c>0$. For this we first note that, by \eqref{eq:def-g-delta}, we may choose $\delta \in (0,1)$ sufficiently small so that \eqref{bdd-above-gs} holds for $s \in [s_0,\frac{1}{4}]$. With this new value of $\delta$ and $\tilde d$ redefined as
$$
\tilde d:= \sup_{s \in [s_0,\frac{1}{4}]} \frac{\|k_{\delta,s}\|_{L^2(\R^N)}}{s} = \sup_{s \in [s_0,\frac{1}{4}]}\frac{C_{N,s}\omega_{N-1}^{\frac{1}{2}}\delta^{-\frac{N}{2}-2s}}{s \sqrt{N+4s}}
< \infty,
$$
we may now fix $c> \tilde d$ and complete the argument as above to see that
also 
$$
\sup \limits_{s \in [s_0,\frac{1}{4}]}\|\phi_s\|_{L^\infty(\Omega)} \le c.
$$
Hence \eqref{eq:uniform-bound-1-third} holds. The proof is now finished by combining \eqref{eq:uniform-bound-s-0} and \eqref{eq:uniform-bound-1-third}. 
\end{proof}

\section{Local equicontinuity}\label{section-Equic}


This section is devoted to prove local equicontinuity of the set $\{\phi_s\::\: s \in (0,\frac{1}{4}]\}$ in $\Omega$.
The first step of the proof consists in deriving $s$-dependent H{\"o}lder estimates for the functions {\em with uniform (i.e., $s$-independent) constants as $s \to 0^+$.} 
As a preliminary tool, we need to consider the {\em Riesz kernel}
\begin{equation}
  \label{eq:def-Riesz-kernel}
F_s:\R^N\setminus \{0\}\to [0,\infty), \qquad 
	F_s(z)= \kappa_{N,s}|z|^{2s-N} \quad \text{with}\quad \kappa_{N,s}=\frac{s\Gamma(\frac{N}{2}-s)}{4^s\pi^{N/2}\Gamma(1+s)}.
\end{equation}

In particular, we need the following Lemma.
\begin{lemma}\label{interior-basis}
	Let $s\in(0,\frac{1}{4}]$, $r \in (0,1)$ and $f\in L^\infty(\overline{B_r})$.
	Moreover, let
        \begin{equation}
          \label{eq:def-u-f}
        u_f: \R^N \to \R, \qquad u_f(x):=\int_{B_r} F_s(x-y)f(y)\ dy.
        \end{equation}
        Then $u_f\in C^{s}(\R^N)\cap L^\infty(\R^N)$, and there is a constant $C=C(N)>0$ such that
        \begin{equation}
          \label{eq:interior-basis-l-infty-est}
	|u_f(x)-u_f(y)| \leq C r^{s}\|f\|_{L^\infty(B_r)}|x-y|^{s} \quad\text{ for all $x,y\in \R^N$.}
        \end{equation}
        If, moreover, $f\in C^{\alpha}(\overline{B_r})$ for some $\alpha\in(0,1-s)$,
        then we also have
         \begin{equation}
          \label{eq:interior-basis-s+alpha-est}
	|u_f(x)-u_f(y)| \leq C r^{s-\alpha} \|f\|_{C^{\alpha}(\overline{B_r})}|x-y|^{s+\alpha} \quad\text{ for all $x,y\in \R^N$}
        \end{equation}
        after making $C=C(N)$ larger if necessary.
      \end{lemma}

      \begin{proof}
For $x\in B_1$ we have
	\[
	u_f(rx)=\int_{B_r} F_s(rx-y)f(y)\ dy=r^{2s}\int_{B_1} F_s(x-z)f(rz)\ dz,
	\]
	so that we may assume $r=1$ in the following. Next, we recall the following standard estimate: 
        \begin{equation}
          \label{eq:bathtub-est}
        \int_{B_t}|x-z|^{\tau-N}\,dz \le \int_{B_t}|z|^{\tau-N}\,dz = \frac{\omega_{N-1}t^\tau }{\tau} \qquad \text{for every $t>0$, $\tau \in (0,N)$ and $x \in \R^N$.}
        \end{equation}
        From this we deduce that $u_f\in L^{\infty}(\R^N)$ with 
        \begin{align}
          \|u_f\|_{L^\infty(\R^N)} &\leq \|f\|_{L^\infty(B_1)} \kappa_{N,s} \sup_{x \in \R^N} \int_{B_1} |x-y|^{2s-N} \,dy \le \|f\|_{L^\infty(B_1)} \frac{\kappa_{N,s} \omega_{N-1}}{2s} \nonumber \\
          &= \frac{\Gamma(\frac{N}{2}-s)\omega_{N-1}}{2^{2s+1}\pi^{N/2}\Gamma(1+s)}\|f\|_{L^\infty(B_1)} \leq C_1 \|f\|_{L^\infty(B_1)} \label{u-f-boundedness-est}
	\end{align}         
	with a constant $C_1=C(N)$ independent of $s \in (0,\frac{1}{4}]$. Next, by e.g. \cite[Eq. (A.3)]{RefMMF}) we use
        \begin{equation}
          \label{eq:lieb-loss-est}
          |a^{2s-N}-b^{2s-N}|\leq \frac{N-2s}{N-s}|a-b|^s(a^{s-N}+b^{s-N})\le
|a-b|^s(a^{s-N}+b^{s-N})\quad \text{for $a,b>0$.}
        \end{equation}
With this estimate and \eqref{eq:bathtub-est}, we deduce that
	\begin{align*}
          &|u(x+h)-u(x)|=\Bigg|\int_{B_1} \bigl(F_s(x-z+h)-F_s(x-z)\bigr) f(z)\ dz\Bigg|\\
          &\leq |h|^{s}\|f\|_{L^\infty(B_1)} \kappa_{N,s}\int_{B_1}(|x-z-h|^{s-N}+|x-z|^{s-N})\ dz\\
&\leq   \frac{2 \omega_{N-1} \kappa_{N,s}}{s} \|f\|_{L^\infty(B_1)}|h|^{s} = \frac{2  \omega_{N-1}\Gamma(\frac{N}{2}-s)}{4^s\pi^{N/2}\Gamma(1+s)} \|f\|_{L^\infty(B_1)}|h|^{s} \qquad \text{for $x,h\in \R^N$.}
	\end{align*}
	Hence there is $C_2=C_2(N)$ independent of $s \in (0,\frac{1}{4}]$ such that 
\begin{equation}\label{what-we-get}
	|u(x+h)-u(x)| \leq C_2 \|f\|_{L^\infty(B_1)}|h|^{s}\qquad \text{for all $x,h \in \R^N$.}
      \end{equation}
      We thus deduce \eqref{eq:interior-basis-l-infty-est}.

      Next we assume that $f\in C^{\alpha}(\overline{B_1})$ for some $\alpha\in(0,1-s)$, and we establish \eqref{eq:interior-basis-s+alpha-est} in the case $r=1$.

We choose a cut-off function $\eta\in C^{\infty}_c(\R^N)$ with $0 \le \eta \le 1$, $\eta\equiv 1$ on $B_{7/8}$ and $\eta=0$ on $\R^N\setminus B_1$.
        We then define $w \in C^\alpha_c(\R^N)$ by $w(x)=\eta(x)f(x)$ for $x\in B_1$ and $w(x)=0$ for $x\in \R^N\setminus B_1$. Then $u_f(x)=u_1(x)+u_2(x)$ for $x\in B_1$ with
	$$
          u_1(x)=\int_{B_1}F_s(x-z)(1-\eta(z))f(z)\ dz = \int_{B_1\setminus B_{7/8}}F_s(x-z)(1-\eta(z))f(z)\ dz
          $$
          and
          $$
          u_2(x)=\int_{\R^N} F_s(x-z)w(z)\ dz \qquad \text{for $x \in \R^N$.}
	$$
Since $|x-z|\geq \frac{1}{8}$ for $x\in B_{3/4}$ and $z\in B_1\setminus B_{7/8}$, for all $\beta\in \N_0^d$, $|\beta|\leq 1$ we have
	\begin{align*}
          |\partial^{\beta}u_1(x)|&=\Big|\int_{B_1}\partial^{\beta}_xF_s(x-z)(1-\eta(z))f(z)\ dz\Big| \le  \|f\|_{L^\infty(B_1)} \|\partial^{\beta}F_s\|_{L^1(B_2 \setminus B_{\frac{1}{8}})}  
          \\
	&\leq \|f\|_{L^\infty(B_1)} \kappa_{N,s}\omega_{N-1}\Bigl((N-2s)\int_{1/8}^{2} t^{2s-2}\ dt+\int_{1/8}^{2} t^{2s-1}\ dt\Bigr) \\
                                  &\leq \|f\|_{L^\infty(B_1)}  \kappa_{N,s} \omega_{N-1} (N+2)\int_{1/8}^{2} t^{2s-2}\ dt \leq C_3 \|f\|_{L^\infty(B_1)}
	\end{align*}
for $x \in B_1$, $s \in (0,1)$ with a constant $C_3=C_3(N)>0$. Hence $u_1\in C^1(\overline{B_{3/4}})$, and 
\begin{equation}
  \label{eq:u-1-estimate}
	|u_1(x)-u_1(y)|\leq C_3 \|f\|_{L^\infty(B_1)} |x-y| \quad\text{ for all $x,y\in B_{3/4}$.}
\end{equation}
To estimate $u_2$, we first note that, by the same estimate as in \eqref{u-f-boundedness-est}, we find that
\begin{equation}
  \label{eq:u_2-boundedness-est}
          \|u_2\|_{L^\infty(B_1)} \le C \|w\|_{L^\infty(B_1)} \le C \|f\|_{L^\infty(B_1)}.  
\end{equation}
Moreover, we write $\delta_hw(x)=w(x+h)-w(x)$ for $x,h\in \R^N$. Since $w$ has a compact support contained in $B_1$ and $\eta$ is smooth, there is $C_4=C_4(N)$ such that
	$$
	|\delta_hw(x)|\leq C_4 \|f\|_{C^{\alpha}(\overline{B_1})} |h|^{\alpha}\quad\text{for all $x,h\in \R^N$.}
	$$
For $x,h\in \R^N$, $|h| \le 1$ we now have, by \eqref{eq:lieb-loss-est} and since $\delta_h w$ is supported in $B_2$, 
	\begin{align*}
          &|u_2(x+2h)-2u_2(x+h)+u_2(x)|\\
          &= \big|\delta^2_hu_2(x)\big|=\Bigg|\int_{\R^N}\delta_hF_s(x-z)\delta_hw(z)\ dz\Bigg|= \Bigg|\int_{B_2}\delta_hF_s(x-z)\delta_hw(z)\ dz\Bigg|\\
          &\leq |h|^{\alpha+s}C_4 \|f\|_{C^{\alpha}(\overline{B_1})} \kappa_{N,s}\int_{B_2}(|x-z-h|^{s-N}+|x-z|^{s-N})\ dz.
        \end{align*}
        Using now \eqref{eq:bathtub-est} again, we deduce that 
	\begin{align*}
          &|u_2(x+2h)-2u_2(x+h)+u_2(x)|\\	&\leq   \frac{\kappa_{N,s}}{s} C_4  \omega_{N-1}2^{s+1}  \|f\|_{C^{\alpha}(\overline{B_1})}|h|^{\alpha+s} = \frac{C_4 \omega_{N-1}2^{s+1} \Gamma(\frac{N}{2}-s)}{4^s\pi^{N/2}\Gamma(1+s)} \|f\|_{C^{\alpha}(\overline{B_1})}|h|^{\alpha+s}.
	\end{align*}
	Hence there is $C_5=C_5(N)$ such that 
\begin{equation}\label{what-we-get2}
	|u_2(x+2h)-2u_2(x+h)+u_2(x)| \leq C_5 \|f\|_{C^{\alpha}(\overline{B_1})} |h|^{\alpha+s} \qquad \text{for all $x\in \R^N$, $|h| \le 1$.}
      \end{equation}
By \eqref{eq:u_2-boundedness-est}, we may make $C_5>0$ larger if necessary so that \eqref{what-we-get2} holds for all $x, h \in \R^N$.      
Since $\alpha+s<1$ by assumption, it now follows, by a well known argument, that
\begin{equation}\label{what-we-get-follow}
	|u_2(x+h)-u_2(x)| \leq C_6 \|f\|_{C^{\alpha}(\overline{B_1})} |h|^{\alpha+s} \qquad \text{for all $x, h \in \R^N$}
\end{equation}
with a constant $C_6=C_6(N)>0$. For the convenience of the reader, we recall this argument in the appendix.
The estimate \eqref{eq:interior-basis-s+alpha-est} now follows by combining \eqref{eq:u-1-estimate} and \eqref{what-we-get-follow}.  
\end{proof}

We need another lemma.

\begin{lemma}
  \label{lemma-jarohs-saldana-weth}
  Let $r>0$, $f \in L^\infty(B_r)$, and suppose that $u \in L^\infty(\R^N)$ is a distributional solution of the equation $(-\Delta)^s u = f$ in $B_r$ for some $s \in (0,\frac{1}{4}]$. Moreover, let $u_f: \R^N \to \R$ be defined as in \eqref{eq:def-u-f}, and let
 $u_*:= u-u_f$. 

 Then we have the estimate
 \begin{equation}
   \label{eq:jarohs-saldana-weth-est}
 |u_*(x)-u_*(y)| \le C |x-y|^{3s} \Bigl(r^{-3s} \|u\|_{L^\infty(\R^N\setminus B_r)} +r^{-s} 
\|f\|_{L^\infty(B_r)}\Bigr)\qquad \text{for $x,y \in B_{\frac{r}{2}}$}  
 \end{equation}
with a constant $C=C(N)>0.$
\end{lemma}

\begin{proof}
  By scaling invariance, it suffices to consider the case $r=1$. In this case, we may follow the proof of \cite[Lemma A.1]{RefJ3}, using the
fact that $u_*$ solves the problem
        $$
        (-\Delta)^s u_* = 0 \quad \text{in $B_r$}\qquad u_* = u-u_f \quad \text{in $\R^N \setminus B_r$.} 
$$        
Using the corresponding Poisson representation of $u_*$, it was shown in \cite[Proof of Lemma A.1]{RefJ3} that 
 \begin{equation}
   \label{eq:jarohs-saldana-weth-est-proof}
   |u_*(x)-u_*(y)| \le c_1 |x-y| \Bigl(\tau_{N,s}\int_{\R^N\setminus B_1} \frac{|u(z)|}{|z|^N(|z|^2-1)^s} dz + 
\|f\|_{L^\infty(B_1)}\Bigr)\qquad \text{for $x,y \in B_{\frac{1}{2}}$}  
 \end{equation}
with a constant $c_1=c_1(N)$ and $\tau_{N,s} =\frac{2}{\Gamma(s)\Gamma(1-s)|S^{N-1}|}$, see \cite[P. 48]{RefJ3}. From this, we deduce \eqref{eq:jarohs-saldana-weth-est} in the case $r=1$ since $s \in (0,\frac{1}{4}]$.  
\end{proof}

\begin{cor}\label{interior-3s}
	Let $s\in(0,\frac{1}{4}]$. Then $\phi_s\in C^{3s}(\overline{B_{r/8}(x_0)})$ for all $x_0\in \Omega$ and $0 < r \le \min \{1,\delta_{\Omega}(x_0)\}$. Moreover, there is $C=C(N,\Omega,k)>0$ such that
	\[
	\sup_{x,y\in B_{r/8}(x_0)}\frac{|\phi_s(x)-\phi_s(y)|}{|x-y|^{3s}} \leq C r^{3s} \qquad \text{for $s \in (0,\frac{1}{4}]$.}
	\]
      \end{cor}
      
\begin{proof}
By translation invariance we may assume $x_0=0\in\Omega$. Let $r \in (0, \min \{1,\delta_{\Omega}(0)\})$. We write $\phi_s=u_{s,1}+u_{s,2}$ with 
	\[
	u_{s,1}(x)=\int_{B_{r}}F_s(x-z)\lambda_s \phi_s(z)\ dz, \quad \text{for $x \in \R^N$,} \qquad u_{s,2}=\phi_s-u_{s,1},
	\]
        where $F_s$ is the Riesz kernel defined in Lemma \ref{interior-basis}. Moreover, in the following, the letter $C>0$ denotes different constants depending only on $N,\Omega$ and $k$.  
        By Lemma~\ref{uniform-l-infty-section} and Lemma~\ref{interior-basis}, we have 
$$
	|u_{s,1}(x)-u_{s,1}(y)|\leq C  r^s |x-y|^{s} \quad\text{ for all $x,y\in \R^N$.}
        $$
  Moreover, by Lemma~\ref{lemma-jarohs-saldana-weth} we have
  \begin{equation}
    \label{eq:u2-3s-est}
    |u_{s,2}(x)-u_{s,2}(y)|\leq C r^{-3s} |x-y|^{3s}
    \leq C r^{-2s} |x-y|^{2s} \le  C r^{-s}|x-y|^{s}\quad\text{for all $x,y\in B_{r/2}$.}
  \end{equation}
Hence
$$
|\phi_s(x)-\phi_s(y)| \leq C r^{-s} |x-y|^{s} \quad\text{ for all $x,y\in B_{r/2}$.}
$$
Applying now the second claim in Lemma~\ref{interior-basis} with $\alpha=s$, we deduce that
$$
	|u_{s,1}(x)-u_{s,1}(y)|\leq C  r^{-s} |x-y|^{2s} \quad\text{ for all $x,y\in B_{r/4}$.}
$$
Combining this estimate with \eqref{eq:u2-3s-est}, we deduce that
$$
|\phi_s(x)-\phi_s(y)| \leq C r^{-2s} |x-y|^{2s} \quad\text{ for all $x,y\in B_{r/4}$.}
$$
Finally, applying the second claim in Lemma~\ref{interior-basis} with $\alpha=2s$, we deduce that
$$
	|u_{s,1}(x)-u_{s,1}(y)|\leq C  r^{-2s} |x-y|^{3s} \quad\text{ for all $x,y\in B_{r/8}$.}
$$
Combining this estimate with \eqref{eq:u2-3s-est}, we deduce that
$$
|\phi_s(x)-\phi_s(y)| \leq C r^{-3s} |x-y|^{3s} \quad\text{ for all $x,y\in B_{r/8}$,}
$$
as claimed.
\end{proof}

We now state a key local bound related to an auxiliary integral operator. 

\begin{lemma}\label{interior-basis2}
  Let $t_{0},r>0$. Then there exists a constant
$C=C(N,\Omega,k,r,t_0)>0$ with the property that 
\[
\Bigl| \int_{B_{t_0}}\frac{\phi_s(x)-\phi_s(x+y)}{|y|^{N+2s}}\,dy\Bigr| \le  C \qquad\text{ for all $s \in (0,\frac{1}{4}]$ and all $x \in \Omega$ with $\delta_\Omega(x) > r$.}
\]
\end{lemma}
\begin{proof}
Without loss of generality, we may assume that $r<1$. Moreover, we fix $x \in \Omega$ with $\delta_\Omega(x) < r$. In the following, we fix $t= \min \{\frac{t_0}{2},\frac{r}{8}\}<1$, and we write 
$$
\int_{B_{t_0}}\frac{\phi_s(x)-\phi_s(x+y)}{|y|^{N+2s}}\,dy = \int_{B_t}\frac{\phi_s(x)-\phi_s(y)}{|y|^{N+2s}}dy
                       - \int_{B_{t_0} \setminus B_t} \frac{\phi_s(y)}{|y|^{N+2s}}dy
                       + \omega_{N-1}\frac{t^{-2s}-{t_0}^{-2s}}{2s} \phi_s(x)
$$                       
and
$$
(-\Delta)^s \phi_s(x)= C_{N,s} \int_{B_t}\frac{\phi_s(x)-\phi_s(y)}{|y|^{N+2s}}dy - C_{N,s} \int_{\R^N \setminus B_t} \frac{\phi_s(y)}{|y|^{N+2s}}dy + \frac{\omega_{N-1} C_{N,s}}{2s}t^{-2s}\phi_s(x). 
$$
Since $C_N \omega_{N-1}=2$, we can thus write 
\begin{equation}\label{bdd-i}
C_N \int_{B_{t_0}}\frac{w(x)-w(x+y)}{|y|^{N+2s}}\,dy-\Bigl(\frac{(-\Delta)^s-1}{s}\Big)\phi_s(x)=I^s_1(x)+I^s_2(x)+I^s_3(x)
\end{equation}
with
\begin{align*}
I^s_1(x)&:= \Bigl(C_N-\frac{C_{N,s}}{s}\Bigr) \int_{{B_{t}}}\frac{\phi_s(x)-\phi_s(x+y)}{|y|^{N+2s}}\ dy\\
  I^s_2(x)&:=\Bigl(\frac{C_{N,s}}{s}-C_N\Bigr)
          \int_{B_{t_0} \setminus {B_{t}}}\frac{\phi_s(x+y)}{|y|^{N+2s}}\ dy \:+\:
          \frac{C_{N,s}}{s} \int_{\R^N  \setminus B_{t_0}}\frac{\phi_s(x+y)}{|y|^{N+2s}}\ dy
\quad \text{and}\\
  I^s_3(x)&:=\frac{\phi_s(x)}{s} \Big(C_N \omega_{N-1}\frac{t^{-2s}-{t_0}^{-2s}}{2}+1
          - \frac{\omega_{N-1} C_{N,s}}{2s}t^{-2s} \Big)
          =
          \frac{\phi_s(x)}{s} \Bigl[\Bigl(1-\frac{C_{N,s}}{C_N s}\Bigr)t^{-2s} + 1-t_0^{-2s}\Bigr].  
\end{align*}
By \eqref{eq:C-N-D-N-ineq} and since 
$$
t^{-2s} \le t^{-\frac{1}{2}}\quad\text{and}\quad \bigl|\frac{1-t_0^{-2s}}{s}\bigr| \le \frac{|\ln t_0|}{2}\max\{1,t_0^{-2s}\} \le \frac{|\ln t_0|}{2}\max\{1,t_0^{-\frac{1}{2}}\}\qquad \text{for $s \in (0,\frac{1}{4}]$,}
$$
it follows that
\begin{equation}\label{bdd-i3}
  |I^s_3(x)|\leq 
\Bigl[D_N t^{-\frac{1}{2}} +\frac{|\ln t_0|}{2}\max\{1,t_0^{-\frac{1}{2}}\}\Bigr] \sup_{s \in (0,\frac{1}{4}]} \|\phi_s\|_{L^\infty(\Omega)},
\end{equation}
where the RHS is a finite constant by Theorem~\ref{uniform-l-infty-section}.
To estimate $I^s_2$, we let $R:=1+\diam(\Omega)$ and note that, by \eqref{eq:C-N-D-N-ineq}, Theorem~\ref{uniform-l-infty-section}, and since $\phi_s \equiv 0$ on $\Omega^c$, 
\begin{align*}
  |I^s_2(x)| &\le \Bigl(\bigl|\frac{C_{N,s}}{s}-C_N\bigr| + \frac{C_{N,s}}{s} \Bigr)\int_{B_R \setminus B_t}\frac{|\phi_s(x+y)|}{|y|^{N+2s}}\,dy\\
  &\le \Bigl( \bigl|\frac{C_{N,s}}{s}-C_N\bigr| + \frac{C_{N,s}}{s} \Bigr)\omega_{N-1}\frac{t^{-2s}-R^{-2s}}{2s}\|\phi_s\|_{L^\infty(\Omega)}\\
             &= \bigl(\bigl|\frac{C_{N,s}}{sC_N}-1\bigr| + \frac{C_{N,s}}{sC_N} \bigr)\frac{t^{-2s}-R^{-2s}}{s}\|\phi_s\|_{L^\infty(\Omega)} \le (2 sD_N  +1)\frac{t^{-2s}-R^{-2s}}{s}\|\phi_s\|_{L^\infty(\Omega)}\\
  &\le \bigl(\frac{D_N}{2}+1\bigr) \: \frac{t^{-2s}-R^{-2s}}{s}\|\phi_s\|_{L^\infty(\Omega)} \qquad \text{for $s \in (0,\frac{1}{4}]$.} 
\end{align*}
Since $\bigl(t^{-2s}-R^{-2s}\bigr)= 2(\ln R- \ln t)s +o(s)$ as $s \to 0^+$, it follows that
\begin{equation}
  |I^s_2(x)| \le \bigl(\frac{D_N}{2}+1\bigr) \sup_{s \in (0,\frac{1}{4}]}\frac{t^{-2s}-R^{-2s}}{s}  \sup_{s \in (0,\frac{1}{4}]} \|\phi_s\|_{L^\infty(\Omega)},
\end{equation}
where the RHS is a finite constant depending on $t$ but not on $s$.

Finally, to estimate $I^s_1(x)$, we note that our choice of $t= \min\{\frac{t_0}{2}, \frac{r}{8}\}$ allows us to apply Lemma~\ref{interior-3s}, which gives that
$$
|\phi_s(x+h)-\phi_s(x)| \leq \tilde C |y|^{3s} \qquad \text{for $s \in (0,\frac{1}{4}]$, $y \in B_t$}
$$
with a constant $\tilde C=\tilde{C}(N,\Omega,k)>0$. Using this together with \eqref{eq:C-N-D-N-ineq} we may estimate
$$
  |I^s_1(x)|\leq \Bigl|C_N-\frac{C_{N,s}}{s}\Bigr| \tilde C \int_{B_t}|y|^{s-N} dy\le \omega_{N-1} \tilde C \bigl(sC_N D_N  \bigr)\frac{t^s}{s} =  2 \tilde C D_N t^s \le 2 \tilde C D_N \quad \text{for $s \in (0,\frac{1}{4}]$.}
$$ 
Going back to \eqref{bdd-i}, we now find that 
$$
\sup_{s \in (0,\frac{1}{4}]}\Bigl|C_N \int_{B_{t_0}}\frac{w(x)-w(x+y)}{|y|^{N+2s}}\,dy-\Bigl(\frac{(-\Delta)^s-1}{s}\Big)\phi_s(x)\Bigr|< \infty.
$$
Since also
$$
\sup_{s \in (0,\frac{1}{4}]}\Bigl\|\Big(\frac{(-\Delta)^s-1}{s}\Big)\phi_s(x)\Bigr\|_{L^\infty(\Omega)} = \sup_{s \in (0,\frac{1}{4}]}\Bigl( \Bigl|\frac{\lambda_s-1}{s}\Bigr|
\bigl\|\phi_s(x) \bigr\|_{L^\infty(\Omega)}\Bigr) <\infty
$$
by Theorems~\ref{lambda-limit-lower-bound} and ~\ref{uniform-l-infty-section}, the claim now follows.
\end{proof}

We now have all tools to complete the proof of Theorem~\ref{main-theorem-introduction}(iii) which we restate here for the reader's convenience.

\begin{thm}
\label{thm-sec-local-equicontinuity}  
The set $\{\phi_{s}\::\: s \in (0,\frac{1}{4}]\}$ is
equicontinuous in every point $x_0 \in \Omega$ and therefore relative compact in $C(K)$ for every compact subset $K \subset \Omega$.  
\end{thm}

\begin{proof}
  We only have to prove the equicontinuity of the set $M:= \{\phi_s\::\: s \in (0,\frac{1}{4}]\}$ in every point $x_0 \in \Omega$. Once this is shown, it follows from Theorem~\ref{uniform-l-infty-section} and the Arzela-Ascoli Theorem that, for every compact subset $K \subset \Omega$, the set $M$ is relative compact when regarded as a subset of $C(K)$.  

Arguing by contradiction, we now assume that there exists a point $x_0 \in \Omega$ such that $M$ is not equicontinuous at $x_0$, which means that 
  \begin{equation}
    \label{eq:positive-osc-limit}
\lim_{t \to 0^+} \sup_{s \in (0,\frac{1}{4}]} \underset{B_t(x_0)}{\osc} \phi_s= \eps>0.
  \end{equation}
Here, we note that this limit exists since the function
$$
(0,\infty) \to [0,\infty), \qquad t \mapsto \sup_{s \in (0,\frac{1}{4}]} \underset{B_t(x_0)}{\osc} \phi_s
$$
is nondecreasing. Without loss of generality, to simplify the notation, we may assume that $x_0 = 0 \in \Omega$.
We first choose $\delta>0$ sufficiently small so that
\begin{equation}
  \label{eq:delta-small-condition}
\frac{\eps-\delta}{2^{N+2}}-2 \cdot 3^{N} \delta > 0
\end{equation}
The relevance of this condition will become clear later. Moreover, we choose $t_0>0$ sufficiently small so that
\begin{equation}
  \label{eq:3t-0-est}
B_{3t_0} \subset \Omega  
\end{equation}
and 
\begin{equation}
    \label{eq:positive-osc-limit-cons}
\eps \le \sup_{s \in (0,\frac{1}{4}]} \underset{B_t}{\osc}\, \phi_s \le \eps +\delta  \qquad \text{for $0 < t \le 2t_0$.}
\end{equation}
By Lemma~\ref{interior-basis2} and \eqref{eq:3t-0-est}, there exists a constant $C_1>0$ with the property that 
\begin{equation}
\label{bound-diff-operator}  
  \Bigl|  \int_{B_{t_0}} \frac{\phi_s(x)-\phi_s(x+ y)}{|y|^{N+2s}}dy\Bigr| \le C_1 \qquad \text{for all $x \in B_{t_0}$, $s \in (0,\frac{1}{4}].$}
\end{equation}
Next, we choose a sequence of numbers $t_n \in (0,\frac{t_0}{5})$ with $t_n \to 0^+$ as $n \to \infty$. By \eqref{eq:positive-osc-limit-cons}, there exists a sequence $(s_n)_n  \subset (0,\frac{1}{4}]$ such that
\begin{equation}
  \label{eq:eps-half-lower-bound}
\underset{B_{t_n}}{\osc}\, \phi_{s_n} \ge \eps-\delta  \qquad \text{for all $n \in \N$,}
\end{equation}
whereas, by Lemma~\ref{interior-3s}, we have
$$
\underset{B_{t_n}}{\osc}\, \phi_{s_n} \le C_2 (2t_n)^{3s_n} \qquad \text{for all $n \in \N$ with a constant $C_2>0$.}
$$
Hence,
\begin{equation}
  \label{eq:t_n-power-s-n-est}
t_n^{s_n} \ge 2^{-s_n} \Bigl(\frac{1}{C_2} \underset{B_{t_n}}{\osc}\, \phi_{s_n}\Bigr)^{\frac{1}{3}} \ge 2^{-\frac{1}{4}} \Bigl( \frac{\eps-\delta}{C_2}\Bigr)^{\frac{1}{3}}  \qquad \text{for all $n \in \N$}
\end{equation}
    which implies, in particular, that
    \begin{equation}
      \label{eq:s-n-to-zero}
      s_n \to 0 \qquad \text{as $n \to \infty$.}
    \end{equation}
    To simplify the notation, we now set $\phi_n := \phi_{s_n}$. By \eqref{eq:eps-half-lower-bound}, we may write 
    \begin{equation}
      \label{eq:introduce-d-n}
\phi_n(\overline{B_{t_n}}) = [d_n-r_n,d_n+r_n] \qquad \text{for $n \in \N$ with some $d_n \in \R$ und $r_n \ge \frac{\eps-\delta}{2}$.}
    \end{equation}
Together with \eqref{eq:positive-osc-limit-cons} and the fact that $\overline{B_{t_n}} \subset B_{2t_0}$, we deduce that 
    \begin{equation}
      \label{eq:B-t-0-interval-est}
\phi_n(B_{2t_0}) \subset [d_n-\frac{\eps-3\delta}{2}\:,\:d_n+\frac{\eps+3\delta}{2}].
    \end{equation}
Moreover, we let
$$
c_n:= \int_{B_{t_0}\setminus B_{3t_n}}|y|^{-N-2s_n}\,dy = \omega_{N-1} \frac{(3t_n)^{-2s_n}-t_0^{-2s_n}}{2s_n} \qquad \text{for $n \in \N$,}
$$
and we note that
\begin{equation}
  \label{eq:c-n-infty}
c_n \to \infty\qquad \text{as $n \to \infty$}
\end{equation}
since $c_n \ge \omega_{N-1}\bigl(\log t_0 - \log (3 t_n)\bigr)$ for $n \in \N$ and $t_n\to 0$ for $n\to\infty$. We also put
$$
A_+^n : = \{y \in B_{t_0}\setminus B_{3 t_n}\::\: \phi_n(y) \ge d_n\}\quad \text{and}\quad 
A_-^n : = \{y \in B_{t_0}\setminus B_{3 t_n}\::\: \phi_n(y) \le d_n\}.
$$
Since
$$
c_n \le \int_{A_+^n} |y|^{-N-2s_n}\,dy + \int_{A_-^n} |y|^{-N-2s_n}\,dy \qquad \text{for all $n \in \N$,}
$$
we may pass to a subsequence such that
$$
\int_{A_+^n} |y|^{-N-2s_n}\,dy \ge \frac{c_n}{2} \quad \text{for all $n \in \N$}\qquad \text{or}\qquad
\int_{A_-^n} |y|^{-N-2s_n}\,dy \ge \frac{c_n}{2} \quad \text{for all $n \in \N$.}
$$
Without loss of generality, we may assume that the second case holds (otherwise we may replace $\phi_n$ by $-\phi_n$ and $d_n$ by $-d_n$). 
We then define the Lipschitz function $\psi_n \in C_c(\R^N)$ by
$$
\psi_n (x) = \left \{
  \begin{aligned}
    &2 \delta ,&&\qquad |x| \le t_n\\
    &0,&& \qquad |x| \ge 2 t_n\\
    &\frac{2 \delta}{t_n}(2 t_n - |x|),&& \qquad t_n \le |x| \le 2 t_n. 
  \end{aligned}
\right.
$$
We also let $\tau_n:= \phi_n + \psi_n$ for all $n \in \N$. By \eqref{eq:B-t-0-interval-est}, we have 
$$
\tau_n = \phi_n \le d_n + \frac{\eps+3\delta}{2}\le d_n + r_n +2 \delta \qquad \text{in $B_{2t_0} \setminus B_{2t_n}$.}
$$
Moreover, since $d_n+r_n \in \phi_n(\overline{B_{t_n}})$ by \eqref{eq:introduce-d-n}, we have 
$$
d_n+r_n + 2 \delta  \in \tau_n(\overline{B_{t_n}}) \subset \tau_n(B_{2t_n}).
$$
Consequently, $\max \limits_{B_{2t_0}}\, \tau_n$ is attained at a point $x_n \in B_{2t_n}$  with
$$
\tau_n(x_n) \ge d_n+r_n + 2 \delta 
$$
which implies that 
\begin{equation}
  \label{eq:phi-n-x-n-lower-bound}
\phi_n(x_n) \ge d_n +r_n \ge d_n + \frac{\eps-\delta}{2}.
\end{equation}
By \eqref{bound-diff-operator} and since $B_{3t_n} \subset B_{t_0(x_n)}$ for $n \in \N$ by construction, we have that
\begin{align}
  C_1 &\ge \int_{B_{t_0}} \frac{\phi_n(x_n)-\phi_n(x_n+ y)}{|y|^{N+2s_n}}dy = \int_{B_{t_0}(x_n)} \frac{\phi_n(x_n)-\phi_n(y)}{|x_n-y|^{N+2s_n}}dy \nonumber\\
      &= \int_{B_{3 t_n}} \frac{\phi_n(x_n)-\phi_n(y)}{|x_n-y|^{N+2s_n}}dy +
      \int_{B_{t_0}(x_n) \setminus B_{3 t_n}} \frac{\phi_n(x_n)-\phi_n(y)}{|x_n-y|^{N+2s_n}}dy. \label{C-1-first-est}
\end{align}
To estimate the first integral, we note that, by definition of the function $\psi_n$,
$$
|\psi_n(x)-\psi_n(y)| \le \frac{2 \delta}{t_n}|x-y| \qquad \text{for all $x,z \in \R^N$.}
$$
Moreover, by the choice of $x_n$ we have $\tau_n(x_n) \ge \tau_n(y)$ for all $y \in B_{3t_n}$. Consequently,
\begin{align}
  &\int_{B_{3 t_n}} \frac{\phi_n(x_n)-\phi_n(y)}{|x_n-y|^{N+2s_n}}dy= \int_{B_{3 t_n}} \frac{\tau_n(x_n)-\tau_n(y)}{|x_n-y|^{N+2s_n}}dy-
    \int_{B_{3 t_n}} \frac{\psi_n(x_n)-\psi_n(y)}{|x_n-y|^{N+2s_n}}dy \nonumber\\
    &\ge - \int_{B_{3 t_n}} \frac{\psi(x_n)-\psi(y)}{|x_n-y|^{N+2s_n}}dy
      \ge - \frac{2 \delta}{t_n}\int_{B_{3 t_n}} |x_n-y|^{1-N-2s_n}dy \ge - \frac{2 \delta}{t_n}\int_{B_{3 t_n}} |y|^{1-N-2s_n}dy
      \nonumber\\
  &= - \frac{3^{1-2s_n} \omega_{N-1} 2 \delta t_n^{-2s_n}}{1-2s_n} \ge  - 12 \omega_{N-1} \delta t_n^{-2s_n} \ge -C_3
 \label{C-3-est}   
\end{align}
with a constant $C_3>0$ independent of $n$. Here we used \eqref{eq:bathtub-est} and \eqref{eq:t_n-power-s-n-est}.

To estimate the second integral in \eqref{C-1-first-est} we first note, since $x_n \in B_{2t_n}$, we have that 
$$
2|y| \ge  |y-x_n| \ge \frac{|y|}{3} \qquad \text{for every $n \in \N$ and $y \in \R^N \setminus B_{3 t_n}$.}
$$
Moreover, by \eqref{eq:positive-osc-limit-cons}, \eqref{eq:B-t-0-interval-est}, and \eqref{eq:phi-n-x-n-lower-bound} we have 
$$
\eps +\delta \ge \phi_n(x_n)-\phi_n(y) \ge d_n + \frac{\eps-\delta}{2} - \phi_n(y) \ge -2 \delta \qquad \text{for $y \in B_{t_0}(x_n) \subset B_{2t_0}$.}
$$
Consequently, combining \eqref{C-1-first-est} and \eqref{C-3-est}, using again \eqref{eq:phi-n-x-n-lower-bound}, we may estimate as follows:
\begin{align*}
  C_1 &+ C_3\ge  \int_{B_{t_0}(x_n) \setminus B_{3 t_n}} \frac{\phi_n(x_n)-\phi_n(y)}{|y-x_n|^{N+2s_n}}dy\\
  &\ge   \int_{B_{t_0}(x_n) \setminus B_{3 t_n}} \frac{[\phi_n(x_n)-\phi_n]_+(y)}{|y-x_n|^{N+2s_n}}dy - 2 \delta \int_{B_{t_0}(x_n) \setminus B_{3 t_n}} |y-x_n|^{-N-2s_n}dy  \\
  &\ge   \frac{1}{2^{N+2s_n}}\int_{B_{t_0}(x_n) \setminus B_{3 t_n}} \frac{[\phi_n(x_n)-\phi_n]_+(y)}{|y|^{N+2s_n}}dy - 2 \cdot 3^{N+2s_n} \delta \int_{B_{t_0}(x_n) \setminus B_{3 t_n}} |y|^{-N-2s_n}dy  \\
           &\ge   \frac{1}{2^{N+2s_n}}\Bigl(\int_{B_{t_0} \setminus B_{3 t_n}} \frac{[\phi_n(x_n)-\phi_n]_+(y)}{|y|^{N+2s_n}}dy - \int_{B_{t_0} \setminus B_{t_0}(x_n)} \frac{[\phi_n(x_n)-\phi_n]_+(y)}{|y|^{N+2s_n}}dy\Bigr)\\
 &\qquad- 2 \cdot 3^{N+2s_n} \delta \Bigl( \int_{B_{t_0} \setminus B_{3 t_n}} |y|^{-N-2s_n}dy +\int_{B_{t_0}(x_n) \setminus B_{t_0}} |y|^{-N-2s_n}dy \Bigr)  \\
           &\ge   \frac{1}{2^{N+2s_n}}\Bigl(r_n \int_{A_n^-}|y|^{-N-2s_n} dy - (\eps+\delta)\int_{B_{t_0} \setminus B_{t_0}(x_n)} |y|^{-N-2s_n}dy\Bigr) \\
&\qquad- 2 \cdot 3^{N+2s_n} \delta \Bigl(c_n + \int_{B_{t_0}(x_n) \setminus B_{t_0}} |y|^{-N-2s_n}dy\Bigr)\\
           &\ge   \Bigl(\frac{r_n}{2 \cdot 2^{N+2s_n}}-2 \cdot 3^{N+2s_n} \delta\Bigr)c_n\\
 &\qquad  - \frac{(\eps+\delta)}{2^{N+2s_n}}\!\int_{B_{t_0} \setminus B_{t_0-2t_n}} |y|^{-N-2s_n}dy\:-\: 2\cdot3^{N+2s_n} \delta \!\int_{B_{t_0+2t_n} \setminus B_{t_0}} |y|^{-N-2s_n}dy\\
           &\ge   \Bigl(\frac{\eps-\delta}{2^{N+2+2s_n}}-2 \cdot 3^{N+2s_n} \delta\Bigr)c_n - o(1) = \Bigl(\frac{\eps-\delta}{2^{N+2}}-2 \cdot 3^{N} \delta \,+\, o(1)\Bigr)c_n - o(1) \qquad \text{as $n \to \infty$,}
\end{align*}
where we used \eqref{eq:introduce-d-n}. By our choice of $\delta>0$ satisfying \eqref{eq:delta-small-condition}, we arrive at a contradiction to \eqref{eq:c-n-infty}. 
The proof is thus finished.
\end{proof}

\section{Uniform boundary decay}
\label{sec:unif-bound-decay}

Throughout this section, we assume that $\Omega$ is a bounded Lipschitz domain satisfying a uniform exterior sphere condition. By definition, this means that there exists a radius $R_0>0$ such that for every point $x_* \in \partial \Omega$ there exists a ball $B^{x_*}$ of radius $R_0$ contained in $\R^N \setminus \overline{\Omega}$ and with $\overline {B^{x_*}} \cap \overline \Omega = \{x^*\}$.\\

We first note the following boundary decay estimate. 

\begin{lemma}\label{boundary-decay-ros-oton-serra}
There is a constant $C=C(N,\Omega,k)>0$ such that 
\begin{equation}
  \label{eq:uniform-constants-estimates}
|\phi_s(x)|\leq C\delta^s_{\Omega}(x)\quad\text{for $x\in \R^N$, $s \in (0,\frac{1}{4}]$.}
\end{equation}
\end{lemma}

\begin{proof}
  We note that $\phi_s$ is a weak solution of
  $$
  (-\Delta)^s \phi_s = f_s \quad \text{in $\Omega$,}\qquad \phi_s \equiv 0 \quad \text{in $\Omega^c$,}
  $$
  where the functions $f_s := \lambda_s \phi_s$, $s \in (0, \frac{1}{4}]$ are uniformly bounded in $L^\infty(\Omega)$ by Theorem~\ref{uniform-l-infty-section}. Therefore, the decay estimate in \eqref{eq:uniform-constants-estimates} essentially follows from \cite[Lemma 2.7]{Ros14}, although it is not stated there that the constant $C$ can be chosen independently of $s$. For an alternative proof of the latter fact, see \cite[Appendix]{RefJ3}.
We stress here that the use of radial barrier functions as in \cite{Ros14} and  \cite[Appendix]{RefJ3} only requires a uniform exterior sphere condition and no further regularity assumptions on $\Omega$.   
\end{proof}

For $\delta>0$, we now consider the one-sided neighborhood of the boundary
$$
\Omega^\delta:= \{x \in \Omega\::\: \delta_\Omega(x) < \delta\}
$$

The main result of the present section is the following. 

\begin{thm}
  \label{thm-unif-boundary-decay-section}
  We have
  $$
  \lim_{\delta \to 0^+} \sup_{s \in (0,\frac{1}{4}]}\|\phi_s\|_{L^\infty(\Omega^\delta)} = 0.
  $$
  In other words, for every $\eps>0$, there exists $\delta_\eps>0$ with the property that
  $$
  |\phi_s(x)| \le \eps \qquad \text{for all $s \in (0,\frac{1}{4}]$, $x \in \Omega^{\delta_\eps}$.}
  $$
\end{thm}

The remainder of this section is devoted to the proof of this theorem. We need some preliminaries. In the following, for $s \ge 0$, we let $L^1_s(\R^N)$ denotes the space of locally integrable functions $u: \R^N \to \R$ such that
$$
\|u\|_{L^1_s}:= \int_{\R^N} \frac{|u(x)|}{(1+|x|)^{N+2s}}\,dx < +\infty.
$$
We note that $L^1_s(\R^N) \subset L^1_t(\R^N)$ for $0 \le s < t$. Next, we need the following generalization of \cite[Theorem 1.1]{RefT1}.

\begin{lemma}
  \label{local conv-log-lap}
  Let $A \subset \R^N$ be a compact set, let $U \subset \R^N$ be an open neighborhood of $A$, and let $u \in L^1_0(\R^N)$ be a function with $u \in C^\alpha_{loc}(U)$ for some $\alpha >0$. Then
  $$
  \lim_{s\to0^+}\sup_{x\in A}\Big| \frac{(-\Delta)^s u(x) - u(x)}{s}- \loglap u(x)\Big|=0.
  $$
\end{lemma}

\begin{proof}
In the following, we assume $\alpha< 1$. Moreover, without loss of generality, we may assume that $u \in C^\alpha(U)$, otherwise we replace $U$ by a compact neighborhood $U' \subset U$ of $A$. Next, since $A$ is compact, we may fix $r \in (0,1)$ such that for all $x\in A$ we have $\dist(x,\R^N \setminus U)>r$. For $x\in A$ we split the expression of the logarithmic Laplacian as
\begin{align*}
L_{\Delta}u(x)= C_N\int_{B_r}\frac{u(x)-u(x+y)}{|y|^N}\ dy - C_N\int_{\R^N\setminus B_r}\frac{u(x+y)}{|y|^N}\ dy +u(x)\Big(\int_{B_1\setminus B_r}\frac{C_N}{|y|^N}\ dy +\rho_{N}\Big).
\end{align*}
With $D_{r,N}(s):= \frac{C_{N,s}\omega_{N-1}}{2s}r^{-2s}$ and since $C_N \omega_{N-1}= 2$, this splitting gives rise to the inequality
\begin{align}
\sup_{x\in A}&\Big|\frac{(-\Delta)^s-1}{s}u(x)-L_{\Delta}u(x)\Big|\le\sup_{x\in A}\int_{B_r}\frac{|u(x)-u(x+y)|}{|y|^N}\Big|\frac{C_{N,s}}{s}|y|^{-2s}-C_N\Big|\ dy \nonumber\\
&\qquad +\sup_{x\in A}\int_{\R^N\setminus B_r}\frac{|u(x+y)|}{|y|^N}\Big|\frac{C_{N,s}}{s}|y|^{-2s}-C_N\Big|\ dy
 +\|u\|_{L^{\infty}(A)}\Big|\frac{D_{r,N}(s)-1}{s}-\rho_N+2 \log r\Big|\notag\\
&\leq \|u\|_{C^{\alpha}(U)}I_1(s)+\sup_{x\in A} I_2(s,x)+\|u\|_{L^{\infty}(A)}I_3(s), \label{local conv-log-lap-eq-1}
\end{align}
where 
\begin{align*}
I_1(s)=\int_{B_r}|y|^{\alpha-N}\Big|\frac{C_{N,s}}{s}|y|^{-2s}-C_N\Big|\ dy,\quad I_2(s,x)&=\int_{\R^N\setminus B_r}\frac{|u(x+y)|}{|y|^N}\Big|\frac{C_{N,s}}{s}|y|^{-2s}-C_N\Big|,\quad\text{and}\quad\\
I_3(s)&=\Big|\frac{D_{r,N}(s)-1}{s}-\rho_N+2 \log r\Big|.
\end{align*}
By Lemma~\ref{C-N-s-asymptotics}, we have $\lim \limits_{s \to 0^+}\frac{D_{r,N}(s)-1}{s} = \rho_N - 2\log r$ and therefore
\begin{equation}
  \label{local conv-log-lap-eq-2}
\lim_{s\to0^+}I_3(s)=\lim_{s\to0^+}\Big|\frac{D_{r,N}(s)-1}{s}-\rho_N+2 \log r\Big| =0.
\end{equation}
Moreover, by \eqref{eq:C-N-D-N-ineq}, we have the inequality 
\begin{equation}
  \label{eq:RefJ3-lemma-2-1-before}
\Big|\frac{C_{N,s}}{s}|y|^{-2s}-C_N\Big| \le \Big|\frac{C_{N,s}}{s}-C_N\Big||y|^{-2s} + C_N\Big||y|^{-2s}-1\Big|\le C_N\Bigl( s\, D_N |y|^{-2s}+\Big||y|^{-2s}-1\Big|\Bigr).
\end{equation}
for $y \in \R^N \setminus \{0\}$. Using that $\bigl||y|^{-2s}-1\bigr|\le \frac{4s}{\alpha}\bigl(|y|^{-2s-\frac{\alpha}{2}}+|y|^{\frac{\alpha}{2}}\bigr)$ by \cite[Lemma 2.1]{RefJ3} it follows that
\begin{equation}
  \label{eq:RefJ3-lemma-2-1-before-1}
\Big|\frac{C_{N,s}}{s}|y|^{-2s}-C_N\Big| \le s\, C_N \Bigl(  D_N |y|^{-2s}+ \frac{4}{\alpha}\bigl(|y|^{-2s-\frac{\alpha}{2}}+|y|^{\frac{\alpha}{2}}\bigr)\Bigr)\qquad \text{for $y \in \R^N \setminus \{0\}$.}
\end{equation}
In particular, 
\begin{equation}
  \label{eq:RefJ3-lemma-2-1-before-2}
\Big|\frac{C_{N,s}}{s}|y|^{-2s}-C_N\Big| \le s\, C_N \Bigl(  D_N + \frac{8}{\alpha}\Bigr) |y|^{-2s-\frac{\alpha}{2}}\qquad \text{for $0 < |y| \le r$ }
\end{equation}
and
\begin{equation}
  \label{eq:RefJ3-lemma-2-1-before-2b}
\Big|\frac{C_{N,s}}{s}|y|^{-2s}-C_N\Big| \le  s\, C_N  r^{-2s+\alpha}\bigl( D_N + \frac{8}{\alpha}\bigr)|y|^{\frac{\alpha}{2}}\qquad \text{for $|y| >r$.} 
\end{equation}
Therefore, \eqref{eq:RefJ3-lemma-2-1-before-2} gives
\begin{align}
  \lim_{s\to0^+}I_1(s)&\leq \lim_{s\to0^+} s C_N \Bigl(  D_N + \frac{8}{\alpha}\Bigr)   \int_{B_r}|y|^{\frac{\alpha}{2} -N- 2s} dy = \lim_{s\to0^+}2 s \Bigl(  D_N + \frac{8}{\alpha}\Bigr)  \, \frac{r^{\frac{\alpha}{2}-2s} }{\frac{\alpha}{2}-2s} =0  \label{local conv-log-lap-eq-3}
\end{align}
It remains to consider $I_{2}(s,x)$ for $x\in A$. For this, let $\epsilon>0$ and note that there is $R_0>0$ such that for any $R\geq R_0$ we have
\begin{equation}\label{limit-estimate}
\int_{\R^N \setminus B_R}\frac{|u(y)|}{|y|^N}\ dy\leq \frac{\epsilon}{C_N2^N}.
\end{equation} 
Indeed, this is possible  since $u\in L^1_0$ and thus $\lim\limits_{R\to0}\int\limits_{\R^N \setminus B_R}\frac{|u(y)|}{|y|^N}\ dy=0$. In the following, we fix $R>\max\{R,R_0\}$ such that $B_1(A)\subset B_{R}$. Note that by this choice we have in particular $\sup\limits_{z\in A}|z| \le \frac{R}{2}$. Using \eqref{eq:RefJ3-lemma-2-1-before-2b} we then split for $x\in A$
\begin{align}
  &I_2(s,x)=\int_{\R^N\setminus B_r(x)}\frac{|u(y)|}{|x-y|^N}\Big|\frac{C_{N,s}}{s}|x-y|^{-2s}-C_N\Big|\ dy\nonumber\\
	&\leq sC_Nr^{-2s-\alpha}(D_N+\frac{8}{\alpha})\int_{B_R\setminus B_r(x)}\frac{|u(y)|}{|x-y|^{N-\frac{\alpha}{2}}}\ dy+C_N\int_{\R^N\setminus B_R}\frac{|u(y)|}{|x-y|^{N}}\Big|\frac{C_{N,s}}{sC_N}|x-y|^{-2s}-1\Big|\ dy.  \label{local conv-log-lap-eq-4}
\end{align}    
 To estimate the first integral in this decomposition, we use the fact that $|x-y| \ge r \ge \frac{r}{R+1}(1+|y|)$ for $y \in B_R \setminus B_r(x)$ and therefore
\begin{align}
\int_{B_R\setminus B_r}\frac{|u(y)|}{|x-y|^{N-\frac{\alpha}{2}}}\ dy& \le  \Bigl(\frac{r}{R+1}\Bigr)^{\frac{\alpha}{2}-N} \int_{B_R \setminus B_r(x)}|u(y)| (1+|y|)^{\frac{\alpha}{2}-N}dy \nonumber\\
  &\le   \Bigl(\frac{r}{R+1}\Bigr)^{\frac{\alpha}{2}-N}  (1+R)^{\frac{\alpha}{2}}\|u\|_{L^1_0}\leq (1+R)^Nr^{\frac{\alpha}{2}-N}\|u\|_{L^1_0}.\label{local conv-log-lap-eq-5}
\end{align}
For the second integral in this decomposition, first note that since $|x-y|\geq \max\{1,\frac{|y|}{2}\}$ for $y\in \R^N\setminus B_R$ by \eqref{eq:key-est} we have for $y\in \R^N\setminus B_R$
\[
\Big|\frac{C_{N,s}}{sC_N}|x-y|^{-2s}-1\Big|\leq 1-4^s|y|^{-2s}(1+s\rho_N+o(s))\leq 1 +O(s)\quad\text{for $s\to 0^+$ (uniform in $x$ and $y$).}
\]
Combining this with \eqref{local conv-log-lap-eq-5} in \eqref{local conv-log-lap-eq-4} we find
\begin{align}
 \lim_{s\to0^+}\sup_{x\in A}I_2(s,x)&\leq C_N\sup_{x\in A}\int_{\R^N\setminus B_R}\frac{|u(y)|}{|x-y|^{N}} \ dy\leq C_N2^N\int_{\R^N\setminus B_R}\frac{|u(y)|}{|y|^{N}} \ dy \leq \epsilon.\label{local conv-log-lap-eq-6}
\end{align}

Combining \eqref{local conv-log-lap-eq-2}, \eqref{local conv-log-lap-eq-3}, and \eqref{local conv-log-lap-eq-6}, we get from \eqref{local conv-log-lap-eq-1} 
\begin{align*}
\lim_{s \to 0^+}\sup_{x\in A} \Big|\frac{(-\Delta)^su(x)-u(x)}{s}-L_{\Delta}u(x)\Big| &\le \epsilon.
\end{align*}
    Here, $\epsilon>0$ is chosen arbitrary and this completes the proof of the lemma.
\end{proof}

Next we state a uniform small volume maximum principle. For this we define, for $s \in (0,1)$ and any open set $U \subset \R^N$, the function space
$$
\cV^s(U):= \{u \in L^2_{loc}(\R^N)\::\: \int_{U}\int_{\R^N} \frac{(u(x)-u(y))^2}{|x-y|^{N+2s}}\,dxdy < \infty\}
$$
It is easy to see that the quadratic form 
$$
\cE_s(u,v) = C_{N,s} \int_{\R^N} \int_{\R^N} \frac{(u(x)-u(y))(v(x)-v(y))}{|x-y|^{N+2s}}\,dxdy
$$
is well-defined in Lebesgue sense for $u \in \cV^s(U)$, $v \in \cH^s_0(U)$, see e.g. \cite{RefJ2} and the references therein. If functions $u \in \cV^s(U)$ and $g \in L^2(U)$ are given, we say that $(-\Delta)^s u \ge g$ {\em in $U$ weak sense} if
$$
\cE_s(u,v)-\int_{U}g v\,dx \:\ge\: 0 \qquad \text{for all $v \in \cH^s_0(U)$, $v \ge 0$.}
$$

\begin{remark}
\label{maximum-principle-remark}
  Let $U \subset \R^N$ be an open bounded set. Moreover, let $g \in L^2(U)$, and let $u \in L^1_s(\R^N) \cap L^2_{loc}(\R^N)$ be a function satisfying $u \in C^\alpha(K)$ for a compact neighborhood $K$ of $\overline U$ and, for some $s \in (0,\frac{\alpha}{2})$,
  $$
  (-\Delta)^s u \ge g \qquad \text{in $U$ in pointwise sense.}
  $$
  Then $u \in \cV^s(U)$, and $u$ satisfies $  (-\Delta)^s u \ge g$ also in weak sense. This follows since, under the stated assumptions, we have 
 \begin{equation*}
\int_{U} [(-\Delta)^su]  v\,dx = \cE_s(u,v) \qquad \text{for all $v \in \cH^s_0(U)$.}
\end{equation*}
The latter property follows  easily by Fubini's theorem.
\end{remark}

Our uniform small volume weak maximum principle now reads as follows. 
\begin{prop}
\label{small-volume-stable}  
There exists $\mu_0= \mu_0(N)>0$ with the property that the operators
$$
(-\Delta)^s -\id,\qquad s \in (0,1)
$$
satisfy the following weak maximum principle on every open set $U \subset \R^N$ with $|U| \le \mu_0$:\\[0.2cm]
For every $s \in (0,1)$ and every function $u \in \cV^s(U)$ satisfying
$$
(-\Delta)^s u \ge u \quad \text{in $U$},\qquad u \ge 0 \quad \text{in $\R^N \setminus U$}
$$
we have $u \ge 0$ on $\R^N$.
\end{prop}

\begin{proof}
By \cite[Prop. 2.3]{RefJ2}, it suffices to find $\mu_0>0$ with the property that
  \begin{equation}
    \label{eq:sufficient-property}
\text{$\lambda_{1,s}(U)>1$ for every open set $U \subset \R^N$ with $|U|\le \mu_0$ and every $s \in (0,1),$}   
  \end{equation}
 where $\lambda_{1,s}(U)$ denotes the first Dirichlet eigenvalue of $(-\Delta)^s$ on $U$.

Let $r_0=r_0(N):= 2 e^{\frac{1}{2}( \psi(\frac{N}{2})-\gamma)}$. It then follows from \cite[Section 4]{RefT1} that $\lambda_{1,L}(B_{r_0})>0$
, where $\lambda_{1,L}(B_{r_0})$ denotes the first Dirichlet eigenvalue of $\loglap$ on $B_{r_0}:= B_{r_0}(0)$.

Since
$$
\frac{\lambda_{1,s}(B_{r_0})-1}{s} \to \lambda_{1,L}(B_{r_0}) \qquad \text{as $s \to 0^+$,}
$$
there exists $s_0 \in (0,1)$ with the property that
$$
\lambda_{1,s}(B_{r_0})>1 \qquad \text{for $s \in (0,s_0)$.}   
$$
By the scaling properties of the fractional Laplacian, this also implies that
\begin{equation}
  \label{eq:r-0-s-pos-eigenvalue}
\lambda_{1,s}(B_r)= \Bigl(\frac{r_0}{r}\Bigr)^{2s}\lambda_{1,s}(B_{r_0}) \ge \lambda_{1,s}(B_{r_0})>1 \qquad \text{for $s \in (0,s_0)$, $r \in (0,r_0]$.}   
\end{equation}
To obtain a similar estimate for $s \in [s_0,1)$, we use a lower eigenvalue bound given by Ba$\rm{\tilde{n}}$uelos and Kulczycki. In \cite[Corollary 2.2]{BK}, they proved that 
$$
\lambda_{1,s}(B_1) \ge 2^{2s} \frac{\Gamma(1+s) \Gamma(\frac{N}{2}+s)}{\Gamma(\frac{N}{2})}\qquad \text{for $s \in (0,1)$.}
$$
From this we deduce that 
\begin{equation}
  \label{eq:r-1-s-pos-eigenvalue}
\lambda_{1,s}(B_r) \ge \Bigl(\frac{2}{r}\Bigr)^{2s} \frac{\Gamma(1+s) \Gamma(\frac{N}{2}+s)}{\Gamma(\frac{N}{2})} \ge \Bigl(\frac{2}{r}\Bigr)^{2 s_0} \frac{\Gamma_{min}}{\Gamma(\frac{N}{2})} > 1\quad \text{for $s \in [s_0,1)$ and $0< r \le r_1 $,}
\end{equation}
 where $r_1:=  2 \Bigl(\frac{\Gamma_{min}}{\Gamma(\frac{N}{2})}\Bigr)^{\frac{1}{2s_0}}$ and $\Gamma_{min}>0$ denotes the minimum of the Gamma function on $(0,\infty)$. Setting $r_*:= \min \{r_0,r_1\}$, we thus find, by combining \eqref{eq:r-0-s-pos-eigenvalue} and \eqref{eq:r-1-s-pos-eigenvalue}, that 
\begin{equation}
  \label{eq:combined-s-pos-eigenvalue}
\lambda_{1,s}(B_r)>1 \qquad \text{for $s \in (0,1)$, $r \in (0,r_*]$.}   
\end{equation}
Next, let $\mu_0:= |B_{r_*}|$, and let $U \subset \R^N$ be a nonempty open set with $|U| \le \mu_0$. Moreover, let $r \in (0,r_*]$ with $|B_r|= |U|$. Combining \eqref{eq:combined-s-pos-eigenvalue} and the Faber-Krahn type principle given in \cite[Theorem 5]{BLMH}, we deduce that 
$$
\lambda_{1,s}(U) \ge \lambda_{1,s}(B_r) >1  \qquad \text{for $s \in (0,1)$,}
$$
as required.
\end{proof}

We recall a result from \cite{RefT1} regarding a
radial barrier type function for the logarithmic Laplacian, see \cite[Lemma 5.3, Case $\tau = \frac{1}{4}$]{RefT1}.

\begin{lemma}
\label{sec:regul-bound-decay-lemma-radial}
Let $R \in (0,\frac{1}{2})$. Then there exists $\delta_0 = \delta_0(R)>0$ and a continuous function $V \in L^1_0(\R^N)$ with the following properties:
\begin{itemize}
\item[(i)] $V \equiv 0$ in $B_R$ and $V>0$ in $\R^N \setminus \overline{B_R}$;
\item[(ii)] $V \in C^1_{loc}(\R^N \setminus \overline{B_R})$;
\item[(iii)] $\loglap V(x) \to \infty$ as $|x| \to R$, $|x|>R$.
\end{itemize}
\end{lemma}

In fact, in \cite[Lemma 5.3]{RefT1} it was only stated that $V$ is locally uniformly Dini continuous on $\R^N \setminus \overline{B_R}$ since this was sufficent for the considerations in this paper. However, the construction in the proof of this lemma immediately yields that $V \in C^1_{loc}(\R^N \setminus \overline{B_R})$.

\begin{proof}[Proof of Theorem~\ref{thm-unif-boundary-decay-section} (completed)]

We need some more notation. For $R>0$ and $R_1>R$, we consider the open annulus
$$
A_{R,R_1}:= \{x \in \R^N \:: R<|x| < R_1\} \:\subset\: \R^N
$$
and its translations
$$
A_{R,R_1}(y):= \{x \in \R^N \:: R<|x-y| < R_1\}, \qquad y \in \R^N.
$$
In the following, we let $\partial^i \Omega \subset \partial \Omega$ denote the subset of boundary points $x_{*} \in \partial \Omega$ for which there exists an (inner) open ball $B_{x_*} \subset \Omega$ with $x_* \in \partial B_{x_*}$.
  
  Since $\Omega$ satisfies a uniform exterior sphere condition, there exists a radius $0<R_0< \frac{1}{2}$ such that for every point $x_* \in \partial^i \Omega$ there exists a (unique) ball $B^{x_*}$ of radius $R_0$ contained in $\R^N \setminus \overline{\Omega}$ and tangent to $\partial B_{x_*}$ at $x_*$. Let $c(x_*)$ denote the center of $B^{x_*}$.\\

Applying Lemma~\ref{sec:regul-bound-decay-lemma-radial} with the value $R: =\frac{R_0}{2}$ now yields a function $V \in L^1_0(\R^N)$ such that the properties (i)-(iii) of Lemma~\ref{sec:regul-bound-decay-lemma-radial} are satisfied.\\

We now choose $\delta_0 \in (0,1)$ sufficiently small such that 
$$
|A_{R,R+\delta_0}|<\mu_0, 
$$
where $\mu_0>0$ is given by Proposition~\ref{small-volume-stable}.\\

Next we consider the finite values
$$
m_1:= \sup_{s \in (0,\frac{1}{4}]}\|\phi_s\|_{L^\infty(\Omega)}\qquad \text{and}\qquad
m_2 := \sup_{s \in (0,\frac{1}{4}]} \Bigl\|\frac{\lambda_s-1}{s} \phi_s \Bigr\|_{L^\infty(\Omega)}
$$
By Lemma~\ref{sec:regul-bound-decay-lemma-radial}(iii), we can make $\delta_0>0$ smaller if necessary to guarantee that 
\begin{equation}
  \label{eq:loglap-2-m-2-bound}
\loglap V(x) \ge 2m_2 \qquad \text{in $A_{R,R+\delta_0}$.}
\end{equation}
Next, for $x_* \in \partial^i \Omega$ and $t \in [0,1]$, we consider the point
$$
z(t,x_*):= x_* + (t+R)\frac{c(x_*)-x_*}{|c(x_*)-x_*|} \qquad \text{in $\R^N \setminus \overline \Omega$.}
$$
which lies on the extension of the line segment spanned by the points $x_*$ and $c(x_*)$ beyond $c(x_*)$. By construction, $\overline{B_R(z(t,x_*))}\cap \overline \Omega = \varnothing$ for $t \in (0,1]$, while, for $t \in (0,\delta_0)$, the intersection
$$
\Omega_{t,x_*} := \Omega \cap A_{R,R+\delta_0}(z(t,x_*))= \Omega \cap A_{R+t,R+\delta_0}(z(t,x_*))
$$
is nonempty. Since $\Omega$ is bounded, there exists $R_1>R$ such that
$$
\Omega \subset A_{R,R_1}(z(t,x_*)) \qquad \text{for all $x_* \in \partial^i \Omega$, $t \in (0,\delta_0)$},
$$
which implies that
\begin{equation}
  \label{eq:inclusion-omega-t-x-star}
\Omega \setminus \Omega_{t,x_*} \subset \overline{A_{R+\delta_0,R_1}(z(t,x_*))} \qquad \text{for all $x_* \in \partial^i \Omega$, $t \in (0,\delta_0)$.}
\end{equation}
Next, we define the translated functions
$$
V_{t,x_*} \in L^1_0(\R^N), \qquad V_{t,x_*}(x)= V(x-z(t,x^*)), \qquad \text{$x_* \in \partial \Omega$, $t \in [0,1]$.}
$$
Since $V$ is positive on the compact set $\overline{A_{R+\delta_0,R_1}}$ by Lemma~\ref{sec:regul-bound-decay-lemma-radial}(i), we may choose $c>1$ sufficiently large such that $V \ge \frac{m_1}{c}$ in $\overline{A_{R+\delta_0,R_1}}$
and thus, by \eqref{eq:inclusion-omega-t-x-star}, also 
\begin{equation}
  \label{eq:V-x-star-lower-bound}
V_{t,x_*} \ge \frac{m_1}{c}
\qquad \text{in $\Omega \setminus \Omega_{t,x_*}$ for all
  $x_* \in \partial^i \Omega$, $t \in (0,\delta_0)$.}
\end{equation}
To finish the proof of the theorem, we now let $\eps>0$ be given. Since $V$ is continous and $V \equiv 0$ on $B_R$ by Lemma~\ref{sec:regul-bound-decay-lemma-radial}(i), we may fix $\delta \in (0,\frac{\delta_0}{2})$ such that
\begin{equation}
  \label{eq:V-continuity-upper-bound}
0 \le V \le \frac{\eps}{c} \qquad \text{in $B_{R+2\delta}$.}
\end{equation}
Since $A_{R+\delta,R+\delta_0} \subset \subset \R^N \setminus \overline{B_R}$, we find, as a consequence of Lemma~\ref{local conv-log-lap} and Lemma~\ref{sec:regul-bound-decay-lemma-radial}, that 
$$
\frac{(-\Delta)^s V - V}{s} \to \loglap V \qquad \text{uniformly on $A_{R+\delta,R+\delta_0}\;$ as $\;s \to 0^+$.}
$$
Hence, by \eqref{eq:loglap-2-m-2-bound}, we may fix $s_1 \in (0,\frac{1}{4}]$ with the property that
\begin{equation}
  \label{eq:-Delta-s-V-lower-bound}
\frac{(-\Delta)^s V - V}{s} \ge m_2 \ge \frac{m_2}{c} \qquad \text{on $A_{R+\delta,R+\delta_0}\;$ for $\;s \in (0,s_1)$.}
\end{equation}
We now claim that 
\begin{equation}
  \label{eq:final-claim-boundary-decay}
|\phi_s(x)| \le \eps \qquad \text{for $s \in (0,s_1)$, $\;x \in \Omega^{\delta}$.}  \end{equation}
To show \eqref{eq:final-claim-boundary-decay}, we let $x \in \Omega^{\delta}$, and we let $x_* \in \partial \Omega$ with $\delta_\Omega(x)= |x-x_*|$. By definition, we then have $x_* \in \partial^i \Omega$. Moreover, by contruction we have 
\begin{equation}
  \label{eq:x-x-star-inclusion}
x \in \Omega \cap A_{R+\delta,R+2\delta}(z(\delta,x_*)) \subset B_{R+2\delta}(z(\delta,x_*)).
\end{equation}
We now define $W:= c V_{\delta,x_*} \in L^1_0(\R^N)$. By \eqref{eq:-Delta-s-V-lower-bound}, we then have that
\begin{equation}
  \label{eq:-Delta-s-W-lower-bound}
(-\Delta)^s W   \ge W + s m_2 \qquad \text{in $A_{R+\delta,R+\delta_0}(z(\delta,x^*))$ for $s \in (0,s_1)$.}
\end{equation}
Consequently, in weak sense,
\begin{align}
  (-\Delta)^s \bigl(W \pm \phi_s\bigr) &=(-\Delta)^s W \pm \lambda_{s} \phi_s \ge\bigl(W \pm \phi_s\bigr) + s \bigl(m_2 \pm \frac{\lambda_s-1}{s}\phi_s\bigr) \nonumber\\
                                       &\ge W \pm \phi_s \qquad \text{in $\Omega_{\delta,x_*}=\Omega \cap A_{R+\delta,R+\delta_0}(z(\delta,x^*))$}\label{weak-max-princ-appl-1}  
\end{align}
by the definition of $m_2$. Moreover, it follows from \eqref{eq:V-x-star-lower-bound} and the definition of $m_1$ that  
\begin{equation}
  \label{eq:weak-max-princ-appl-2}
W \pm \phi_s \ge  0 \qquad \text{in $\R^N \setminus
\Omega_{\delta,x_*}$ for $s \in (0,s_1)$.}
\end{equation}
Using Proposition~\ref{small-volume-stable}, \eqref{weak-max-princ-appl-1}, and \eqref{eq:weak-max-princ-appl-2} together with the fact that $|\Omega_{\delta,x_*}| \le |A_{R,R+\delta_0}| \le \mu_0$, we deduce that 
$$
W \pm \phi_s \ge 0 \qquad \text{in $\R^N$,}
$$
and thus, in particular,
$$
|\phi_s| \le W \le \eps \qquad \text{in $B_{R+2\delta}(z(\delta,x_*))$ for $s \in (0,s_1)$}
$$
by \eqref{eq:V-continuity-upper-bound}. Consequently, $|\phi_s(x)| \le \eps$ for $s \in (0,s_1)$ by \eqref{eq:x-x-star-inclusion}, and this yields \eqref{eq:final-claim-boundary-decay}. Making $\delta>0$ smaller if necessary, we may, by Lemma~\ref{boundary-decay-ros-oton-serra}, also assume that
\begin{equation}
  \label{eq:final-claim-boundary-decay-extended}
  |\phi_s(x)| \le \eps \qquad \text{for $s \in [s_1,\frac{1}{4}]$, $\;x \in \Omega^{\delta}$.}
\end{equation}
Combining \eqref{eq:final-claim-boundary-decay} and \eqref{eq:final-claim-boundary-decay-extended}, we conclude that
$$
  |\phi_s(x)| \le \eps \qquad \text{for $s \in (0,\frac{1}{4}]$, $\;x \in \Omega^{\delta}$.}
$$
The proof of Theorem~\ref{thm-unif-boundary-decay-section} is thus finished.
\end{proof}

\section{Completion of the proofs}
\label{sec:proof-main-theorems}
In this section, we complete the proofs of Theorem~\ref{main-theorem-introduction}, Corollary~\ref{first-cor-intro} and Corollary~\ref{cor-regularity-loglap}.

We start with the

\begin{proof}[Proof of Theorem~\ref{main-theorem-introduction}] 
Part (i) is proved in Theorem~\ref{lambda-limit-lower-bound}. Part (iii) is proved in Theorem~\ref{thm-sec-local-equicontinuity}. Moreover, the first claim in Part (ii), the boundedness of the set $M:= \{\phi_{k,s}\::\: s \in (0,\frac{1}{4}]\}$ in $L^\infty(\Omega)$, has been proved in Theorem~\ref{uniform-l-infty-section}.
Combining this fact with the relative compactness of the set $M$ in $C(K)$ for every compact subset $K \subset \Omega$,   it follows from Theorem~\ref{thm-unif-boundary-decay-section} together with the  Kolmogorov--Riesz compactness theorem   that
$M$ is relative compact in $L^p(\Omega)$ for every $p \in [1,\infty)$, this completes the claim in Part (ii).


Part (iv) of Theorem~\ref{main-theorem-introduction} follows by combining Part (iii) with Theorem~\ref{thm-unif-boundary-decay-section}.
In fact, since $\Omega$ satisfies an exterior sphere condition and $\phi_{k,s}\equiv 0$ on $\R^N\setminus \Omega$ for all $k\in \N$,  it follows from Lemma \ref{boundary-decay-ros-oton-serra} that for any $k\in \N$ and $s\in (0,\frac{1}{4}]$, $\phi_{k,s}$ has a unique extension to a continuous function on $\overline{\Omega}$ which we still denote by  $\phi_{k,s}$. Note that $\phi_{k,s}=0$ on $\partial\Omega$ and   $\phi_{k,s} \in C_0(\R^N)$. Therefore, it follows from  Theorem~\ref{thm-unif-boundary-decay-section} that $M$ is bounded in $C_0(\R^N)$. To prove that $M$ is equicontinuous in $C_0(\R^N)$ it is enough to show that $\phi_{k,s}$  is equicontinuous at all points in $\partial\Omega$. But this follows from Theorem \ref{thm-unif-boundary-decay-section} and  Part(iii) of Theorem \ref{main-theorem-introduction}. Therefore  M is bounded and  equicontinuous in $C_0(\R^N)$ and  by the  Ascoli-Arzela Theorem,  $M$ is relative compact in $C_0(\R^N)$.

To prove Part (v), let $(s_n)_n \subset (0,\frac{1}{4}]$ be a sequence of numbers with $s_n \to 0^+$. By Theorem~\ref{lambda-limit-lower-bound}, we may pass to a subsequence with the property that
\begin{equation}
  \label{eq:l-2-convergence-last-section}
\phi_{k,s_n} \to \phi_{k,L} \qquad \text{in $L^2(\Omega)$ as $n \to \infty$.}  
\end{equation}
Due to the relative compactness of the set $M$ in $L^p(\Omega)$ already proved in Part (ii), we also have $L^p$-convergence in \eqref{eq:l-2-convergence-last-section} for $1 \le p < \infty$, and the locally uniform convergence follows from Part (iii). Moreover, in the case where $\Omega$ satisfies an exterior sphere condition, the convergence in $C_0(\Omega)$ follows from the relative compactness in the space $C_0(\Omega)$ stated in Part (iv). 
\end{proof}

Next we complete the

\begin{proof}[Proof of Corollary~\ref{first-cor-intro}] 
For the particular  case   $1\le p \le  2$, the convergent in  \eqref{eq:convergence-first-eigenfunction-cor} follows directly from  \cite[Theorem 1.5]{RefT1} combined with the H\"{o}lder inequality. But using the relative compactness of the set $M$ in $L^p(\Omega)$ proved in Part (ii)  of Theorem~\ref{main-theorem-introduction} and the uniqueness of $\phi_{1,s}$, the $L^p$-convergence in   \eqref{eq:convergence-first-eigenfunction-cor}    for $1 \le p < \infty$ and   the locally uniform convergence in $\Omega$ also  follows by Part (iv) of Theorem~\ref{main-theorem-introduction}. The additional assertion follows from the additional assertion in Theorem~\ref{main-theorem-introduction}(v).
\end{proof}

\begin{proof}[Proof of Corollary~\ref{cor-regularity-loglap}] 
  Let $(s_n)_n \subset (0,\frac{1}{4}]$ be a sequence of numbers with $s_n \to 0^+$. Moreover, for every $n \in \N$, let $\phi_{k,s_n}$, $k \in \N$ denote $L^2$-orthonormal Dirichlet eigenfunctions of $(-\Delta)^{s_n}$ on $\Omega$ corresponding to the eigenvalues $\lambda_{k,s_n}$. Passing to a subsequence, we may assume, by Theorem~\ref{main-theorem-introduction}, that
  \begin{equation}
    \label{eq:proof-of-reg-loglap}
\frac{\lambda_{k,s_n}  -1}{s_n} \to \lambda_{k,L} \qquad \text{and}\qquad 
\phi_{k,s_n} \to \phi_{k,L} \quad \text{in $L^2(\Omega)$}
  \end{equation}
  as $n \to \infty$, where, for every $k \in \N$, $\phi_{k,L}$ is a Dirichlet eigenfunction of $\loglap$ on $\Omega$ corresponding to the eigenvalue $\lambda_{k,L}$. Parts (iii) and (v) of Theorem~\ref{main-theorem-introduction} then imply that
  $$
  \phi_{k,L} \in L^\infty(\Omega) \cap C_{loc}(\Omega) \qquad \text{for every $k \in \N$.}
  $$
  Moreover, it follows that $\phi_{k,L} \in C_0(\Omega)$ in the case where $\Omega$ satisfies an exterior sphere condition. 

  Finally, the $L^2$-convergence in \eqref{eq:proof-of-reg-loglap} implies that the sequence of functions $\phi_{k,L}$, $k \in \N$ is $L^2$-orthonormal. It then follows that every Dirichlet eigenfunction of $\loglap$ on $\Omega$ can be written as a finite linear combination of the functions $\phi_{k,L}$, and therefore it has the same regularity properties as the functions $\phi_{k,L}$, $k \in \N$.
\end{proof}

\appendix
\section{On equivalent H{\"o}lder estimates}

Here we recall that by the notion of H\"older-Zygmund spaces we have for $\tau\in(0,1)$ and $r>0$ that $v\in C^{\tau}(\R^N)\cap L^{\infty}(\R^N)$ if and only if
	\begin{equation}\label{holder-zygmund-remark:eq1}
	\|v\|_{L^{\infty}(\R^N)}+\sup_{\substack{x,h\in \R^N \\ h\neq 0}}\frac{|2v(x+h)-v(x+2h)-v(x)|}{|h|^{\tau}}=:v_{\tau}<\infty.
	\end{equation}
	Indeed, if $v\in C^{\tau}(B_r(0))\cap L^{\infty}(\R^N)$, then clearly \eqref{holder-zygmund-remark:eq1} holds. To see the reverse implication, first note that we have $\|v\|_{L^{\infty}(\R^N)}\leq v_{\tau}<\infty$ by \eqref{holder-zygmund-remark:eq1}. Next, let $x\in \R^N$ and we claim that there is $C_2$ independent of $x$ such that
	\[
	\sup_{\substack{y\in \R^N\\ h\neq 0}}\, \frac{|v(x+h)-v(x)|}{|h|^\tau}\leq C_2.
	\]
	Since $v(x+h)-v(x)=(v-c)(x+h)-(v-c)(x)$ for all constants $c\in \R$, we may assume $v(x)=0$. Next, let $h\in \R^N$, then
	\[
	|2v(x+2^{k}h)-v(x+2^{k+1}h)|=|2v(x+2^{k}h)-v(x+2^{k+1}h)-v(x)|\leq v_{\tau}2^{k\tau}|h|^{\tau}\quad\text{for $k\in \N_0$.}
	\]
	But then, for $n\in \N$ and since $\tau<1$,
	\begin{align*}
	|2^nv(x+h)-v(x+2^{n}h)|&\leq \sum_{k=0}^{n-1}2^{n-1-k}|2v(x+2^{k}h)-v(x+2^{k+1}h)|\\
	&\leq C|h|^{\tau}\sum_{k=0}^{n-1}2^{n-1-k+k\tau}\leq v_{\tau}2^n|h|^{\tau}\sum_{k=0}^{\infty}2^{-(1-\tau)k}=\frac{v_{\tau}2^{n}}{1-2^{\tau-1}}|h|^{\tau}.
	\end{align*}
	Hence, for all $n\in \N$,
	\begin{align*}
	|v(x+h)-v(x)|&=|v(x+h)|\leq 2^{-n}|2^nv(x+h)-v(x+2^nh)|+2^{-n}|v(x+2^nh)|\\
	&\leq \frac{v_{\tau}}{1-2^{\tau-1}}|h|^{\tau}+2^{-n}v_{\tau}
	\end{align*}
	and, for $n\to \infty$, we have $|v(x+h)-v(x)|\leq \frac{v_{\tau}}{1-2^{\tau-1}}|h|^{\tau}$ so that  $v\in C^{\tau}(\R^N)\cap L^{\infty}(\R^N)$.

\bibliographystyle{amsplain}

\end{document}